\documentclass[11pt]{amsart}

\usepackage{amsmath}
\usepackage[inline]{enumitem}
\usepackage{amsthm}
\usepackage{amssymb}
\usepackage{mathtools}
\usepackage{graphicx}
\usepackage{array}
\usepackage{multicol}
\usepackage{cite}

\usepackage{xcolor}
\usepackage{physics}
\usepackage{amsmath}
\usepackage{tikz}
\usepackage{mathdots}
\usepackage{yhmath}
\usepackage{cancel}
\usepackage{array}
\usepackage{multirow}
\usepackage{amssymb}
\usepackage{gensymb}
\usepackage{tabularx}
\usepackage{extarrows}
\usepackage{booktabs}
\usetikzlibrary{fadings}
\usetikzlibrary{patterns}
\usetikzlibrary{shadows.blur}
\usetikzlibrary{shapes}

\newtheorem{innercustomthm}{Theorem}
\newenvironment{mythm}[1]
{\renewcommand\theinnercustomthm{#1}\innercustomthm}
{\endinnercustomthm}

\newtheorem{theorem}{Theorem}
\newtheorem{lemma}{Lemma}

\newtheorem{definition}{Definition}
\newtheorem{question}{Question}
\newtheorem{corollary}{Corollary}

\newcommand{\R}{\mathbb{R}}
\newcommand{\Q}{\mathbb{Q}}

\newcommand{\hdim}{\dim_{\mathcal{H}}}
\newcommand{\mdim}{\dim_{\mathcal{M}}}
\newcommand{\cH}{\mathcal{H}}
\newcommand{\cV}{\mathcal{V}}
\newcommand{\cR}{\mathcal{R}}
\newcommand{\E}{\mathcal{E}}

\DeclareMathOperator{\diam}{diam}
\DeclareMathOperator{\SO}{SO}
\DeclareMathOperator{\supp}{supp}

\DeclarePairedDelimiter\ang{\langle}{\rangle}
\DeclarePairedDelimiter\norma{\Vert}{\rVert}

\title{Trees of Dot Products in Thin Subsets of $\R^d$}
\author{Arian Nadjimzadah}

\begin{document}
	
\begin{abstract}
	A. Iosevich and K. Taylor showed that compact subsets of $\R^d$ with Hausdorff dimension greater than $(d+1)/2$ contain 
	trees with gaps in an open interval. 
	Under the same dimensional threshold, we prove the analogous result where distance is replaced by the dot product. We additionally show that the gaps of embedded trees of dot products are prevalent in a set of positive Lebesgue measure, and 
	for Ahlfors-David regular sets, the number of trees with given gaps agrees with the regular value theorem.
\end{abstract}

\maketitle
	\section{Introduction}
		The theme of this work can be summarized in the following question: how large must a subset of $\R^d$ be for it to contain certain geometric structures? Though in our work we focus on dot products, the the study of such questions was first motivated by distances.
		
		If $E$ is a set in $\R^d$, define its distance set by $\Delta(E) = \{|x-y| : x,y\in E\}$. 
		When $E \subset \R^2$ is finite, the study of the relationship between $|\Delta(E)|$ and $|E|$ is the celebrated Erd\H{o}s distance problem. The conjecture is $|\Delta(E)| \geq |E|/\log |E|$, which was met up to a square root with Guth and Katz's bound of $|\Delta(E)| \geq |E|/\sqrt{\log |E|}$ \cite{guth-katz}. 
		One could ask what happens when $E \subset \R^d$ is infinite. A first notion of size that one learns in real analysis is the Lebesgue measure, which we will denote from here onward by $|\cdot|$.
		The following question could be posed.
		\begin{question}
			If $|E| > 0$, how large must $\Delta(E)$ be? 
		\end{question}
		A theorem of Steinhaus says that when $|E| > 0$, $E - E$ contains an open set around 0, so in particular $\Delta(E)$ contains an open set. This is as large of a set in $\R^d$ that we could ever hope for, so we need a more refined notion of the size of infinite sets. 
		Another notion of size that one might encounter is the Minkowski dimension. 
		\begin{definition}[Minkowski Dimension] \label{d:minkowski}
			Let $N(E, \epsilon)$ be the number of balls of radius $\epsilon > 0$ required to cover the set $E$. Then the lower Minkowski dimension of $E$ is given by 
			\begin{equation*}
				\underline{\mdim}(E) = \liminf_{\epsilon \to 0} \frac{\log N(E, \epsilon)}{\log(1/\epsilon)},
			\end{equation*}
			and the upper Minkowski dimension is 
			\begin{equation*}
				\overline{\mdim}(E) = \limsup_{\epsilon \to 0} \frac{\log N(E, \epsilon)}{\log(1/\epsilon)}.
			\end{equation*}
		\end{definition}
		We can pose the following possibly more refined question.
		\begin{question}
			How large does $\underline{\mdim}(E)$ have to be for $|\Delta(E)| > 0$?
		\end{question}
		Unfortunately this question is still uninteresting. There exist sets 
		which have ``full'' Minkowski dimension, in the sense that the lower Minkowski dimension is as large as it can be, yet their distance sets have measure 0. In fact, they can be merely countable!
		Consider 
		\begin{equation*}
			E = \Q^d \cap [0,1]^d.
		\end{equation*}
		By the density of the rationals, it takes (up to a constant) $1/\epsilon^d$ balls of radius $\epsilon$ to cover $E$, regardless of how small we take $\epsilon$. Thus $\underline{\mdim(E)} = d$, the largest possible dimension in $\R^d$.
		However $\Delta(E)$ is the image of a countable set, so it is itself countable. 
		
		The deficiency in Minkowski dimension is that our covers can consist only of balls of \emph{the same size}. However this does make computations with Minkowski dimension easier.
		The Hausdorff dimension does not have this issue, but it is often more difficult to compute.
		\begin{definition}[Hausdorff Dimension]
			Let
			\begin{equation*}
				\cH_\delta^s(E) = \inf \sum_j r_j^s,
			\end{equation*}
			where the infimum is taken over all countable coverings of $E$ by balls $\{B(x_i, r_i)\}$ such that $r_i < \delta$. Define 
			the $s$-dimensional Hausdorff measure $\cH^s$ by
			\begin{equation*}
				\cH^s(E) = \lim_{\delta \to 0} \cH_\delta^s(E).
			\end{equation*}
			The \emph{Hausdorff Dimension} of $E$, $\hdim(E)$, is the unique number $s_0$ such that $H^s(E) = \infty$ if $s < s_0$ and
			$H^s(E) = 0$ if $s > s_0$.
		\end{definition}

		We can now ask the following interesting question. 
		\begin{question}
			How large must $\hdim(E)$ be to ensure that $|\Delta(E)| > 0$?
		\end{question}
		Kenneth Falconer constructed compact sets $E \subset \R^d$ with $\hdim(E) < d/2$ and $|\Delta(E)| = 0$. He also showed the first nontrivial threshold $\hdim(E) > (d+1)/2]$, which ensures $|\Delta(E)| > 0$ \cite{falconer1985}.
		The correct threshold thus lies in $[d/2, (d+1)/2)$ and the conjecture is $d/2$. 
		The cutting edge is still far from the conjectured threshold. Below is a 
		summary of progress to date.
		\begin{equation*}
			\begin{cases}
				\frac{5}{4}, & d=2, \text{\cite{guth2020falconer}} \\
				\frac{9}{5}, & d =3, \text{\cite{du2021weighted}} \\
				\frac{d}{2} + \frac{1}{4} & d\geq 4, d \text{ even},\text{\cite{du2021improved}} \\
				\frac{d}{2} + \frac{1}{4} + \frac{1}{4(d-1)} & d\geq 4, d \text{ odd},
				\text{\cite{du2019sharp}}
			\end{cases}.
		\end{equation*}

		Steps have been taken in understanding more complex distance configurations in $E$. Let $G$ be a graph and define the \emph{$G$-distance configuration} of $E$ by
		\begin{equation*}
			\Delta_G(E) = \{(|x_i - x_j|)_{(i,j) \in \E(G)} : (x_1,\ldots, x_{|\cV(G)|}) \in E^{|\cV(G)|}\}.
		\end{equation*}
		Here $\cV(G)$ and $\E(G)$ are the vertices and edges of $G$ respectively.
		A. Iosevich and K. Taylor \cite{treesIosevichTaylor} showed that for a tree $T$, $\Delta_T(E)$ contains an entire interval when $E \subset \R^d$ has $\hdim(E) > (d+1)/2$.
		At the other extreme, A. Greenleaf, A. Iosevich, B. Liu and E. Palsson
		\cite{group-actions} showed using a group theoretic approach that if $G$ is the complete graph on $k+1$ vertices (the $k$-simplex), then $|\Delta_G(E)| > 0$ as long as $\hdim(E) > (dk+1)/(k+1)$. 
		
		Distance is certainly not the only quantity that can be associated with two points, and progress has been made in generalizing the Falconer problem in this direction too. A. Greenleaf, A. Iosevich, and K. Taylor \cite{nonempty_interior_radon}
		considered more general $\Phi$-configurations for a class of $\Phi : \R^d \times \R^d \to \R^k$. They showed that the associated configuration set $\Delta_{\Phi}(E) = \{\Phi(x,y) : x,y\in E\}$ has nonempty interior under certain lower bound assumptions on $\hdim(E)$ and regularity of the family of generalized Radon transforms associated with $\Phi$. To avoid some of the Fourier integral operator theory needed to handle a general class of $\Phi$ and because of the nice geometric interpretation, we specialize to dot product in $\R^d$, i.e. 
		\begin{equation*}
			\Phi(x,y) = x \cdot y = x^1y^1 + \cdots x^d y^d.
		\end{equation*}
		Define $\Lambda(E) = \{x \cdot y : x,y\in E\}$.
		The lower bounds on Hausdorff dimension for dot products and similar configurations as in \cite{nonempty_interior_radon} are far less developed than for distances. The best bound so far to ensure that $|\Lambda(E)| > 0$ is $\hdim(E) > (d+1)/2$. Compare this with the table above for distances.
		
		In this work we make progress on understanding $T$-dot-product configurations, for $T$ a tree with some $k$ edges. Define 
		\begin{equation*}
			\Lambda_T(E) = \{(x_i \cdot x_j)_{(i,j) \in E(T)} : (x_1,\ldots, x_{k+1}) \in E^{k+1}\}.
		\end{equation*}
		Before arriving at our results, we need the following machinery. It is well known that if $E\subset \R^d$ has $\hdim(E) > \alpha$, there is a number $s \in (\alpha, \hdim(E))$ and finite Borel measure supported on $E$ such that
		\begin{equation*}
			\mu(B(x,r)) \lesssim r^s,
		\end{equation*}
		for each $x \in \R^d$ and $r >0$. We call such a $\mu$ a \emph{Frostman measure with exponent} $s$. In light of this, we have the following results.
		\begin{theorem}\label{t:upper_bound}
			Let $T$ be a tree with $k$ edges and $E \subset \R^d$ compact with $\hdim(E) > (d+1)/2$. Then
			for every Frostman measure with exponent $s > (d+1)/2$ supported on $E$, there is a constant $C > 0$ independent of $\epsilon$ such that
			\begin{equation}\label{eq:t:upper_bound}
				\mu^{k+1}(\{(x_1,\ldots, x_{k+1}) \in E^{k+1}: t^{ij} - \epsilon < x_i \cdot x_j < t^{ij} + \epsilon, (i,j) \in \E(T) \}) < C\epsilon^k,
			\end{equation}
			for every collection $\{t^{ij}\}$ and $\epsilon > 0$. 
		\end{theorem}
		In the proof of Theorem \ref{t:upper_bound},
		we follow a scheme developed by A. Iosevich et.~al. \cite{iosevich2019maximal} to bootstrap a Sobolev operator bound to a $L^2(\mu) \to L^2(\mu)$ bound. This gives us a mechanism to `rip' leaves from a tree until nothing is left. We remark that in the case of chain configurations
		\begin{equation*}
			\{(x_1 \cdot x_2, x_2 \cdot x_3, \ldots, x_k \cdot x_{k+1}) : (x_1,\ldots, x_{k+1}) \in E^k\},
		\end{equation*}
		Theorem \ref{t:upper_bound} is a special case of work done by A. Iosevich, K. Taylor, and I. Uriarte-Tuero \cite{IosevichTaylorIgnacio}.

		We would also like to find a lower bound for a quantity like \eqref{eq:t:upper_bound}. The idea
		will be to embed $T$ in a \emph{symmetric tree cover} $\sigma(T)$ which can be `folded' down to a
		single edge, at which point we can apply a result in \cite{nonempty_interior_radon} to the single edge.
		We define $\sigma(T)$ in Section \ref{sec:lower_bound}. 
		\begin{theorem}\label{t:lower_bound}
			Let $T$ be a tree with $k$ edges and $E \subset \R^d$ compact with $\hdim(E) > (d+1)/2$. For every Frostman measure with exponent $s > (d + 1)/2$ supported on $E$, there is a constant $c > 0$ independent of $\epsilon$ and open interval $I$ such that for each $t \in I$ and $\epsilon > 0$,
			\begin{equation*}
				\mu^{k+1}(\{(x_1,\ldots, x_{k+1}) \in E^{k+1}: t - \epsilon < x_i \cdot x_j < t + \epsilon, (i,j) \in \E(\sigma(T)) \}) > c\epsilon^k.
			\end{equation*}
		\end{theorem}
		From Theorem \ref{t:upper_bound} we can deduce that any tree $T$ is embedded in $E$ with \emph{many} different edge-wise dot products. We mean this in the following sense.
		\begin{corollary} \label{c:upper}
			Let $E \subset \R^d$ be compact with $\hdim(E) > (d+1)/2$. Then $|\Lambda_T(E)| > 0$.
		\end{corollary}
		It is also interesting to pinpoint which embeddings of a graph are contained in $E$ and how many such embeddings there are.
		For the distance variant of this question see \cite{treesIosevichTaylor}.
		We define the set of embeddings of $T$ in $E$ with dot-product vector $t = (t^{ij})$ as 
		\begin{equation*}\label{eq:T-embedding}
			T_t(E) = \{(x_1, \ldots, x_{k+1}) \in E^{k+1} : x_i \cdot x_j = t^{ij}, (i,j)\in \E(T)\}.
		\end{equation*}
		When $t$ is a scalar, we take all the $t^{ij} = t$ in \eqref{eq:T-embedding}.
		We can use Theorem \ref{t:lower_bound} to show that there are embeddings with equal edge value.
		\begin{corollary}\label{c:equal_embeddings}
			Let $E \subset \R^d$ be compact with $\hdim(E) > (d+1)/2$. Then 
			there is an open interval $I$ such that for each $t \in I$,
			$T_t(E)$ is nonempty.
		\end{corollary}
		Using Theorem \ref{t:upper_bound}, we can show that 
		when $E$ is Ahlfors-David regular,
		there cannot be too many 
		embeddings of any given type $t = (t^{ij})$.
		Before getting to the corollary, we define Ahlfors-David regular. 
		\begin{definition}
			A set $E \subset \R^d$ is Ahlfors-David $s$-regular if it is closed and 
			if there exists a Borel measure $\mu$ supported on $E$ and a constant $C$ such that 
			\begin{equation*}
				C^{-1}r^s \leq \mu(B(x,r)) \leq C r^s,
			\end{equation*}
			for all $x \in E$, $0 < r \leq \diam(E)$, $r < \infty$. 
		\end{definition}
		Note that when working on compact sets $E$, such measures $\mu$ are finite. 
		We prove the following. 
		\begin{corollary}\label{c:regular}
			Let $E \subset \R^d$ be compact Ahlfors-David $s$-regular, for some $s > (d+1)/2$. Then 
			for any $t = (t^{ij})$, 
			\begin{equation*}
				\overline{\mdim}(T_t(E)) \leq (k+1)s - k.
			\end{equation*}
		\end{corollary}
		As we explained above, Minkowski dimension is a weaker notion than Hausdorff dimension when working with lower bounds. However 
		for upper bounds, Minkowski dimension is the stronger statement. In summary
		\begin{equation*}
			\hdim(A) \leq \underline{\mdim(A)} \leq \overline{\mdim(A)}.
		\end{equation*}
		Corollary \ref{c:regular} should not be too surprising. 
		Say we were working on $\R^d$ instead of $E$. Then we have $k$ equations $x_i \cdot x_j = t^{ij}$ and $k+1$ variables $x_1,\ldots, x_{k+1}$.
		So the regular value theorem tells us that $T_t(E)$ has dimension $1$. 
		In our case $E$ has dimension $s$, so one can think of $T_t(E)$ as $s(k+1)$ dimensions of freedom cut by $k$ equations, giving $(k+1)s - k$ remaining dimensions.

	\section{Initial Reductions}\label{sec:reductions}
		Let $E \subset \R^d$ have $\hdim(E) > (d+1)/2$. Then there is a Frostman measure $\mu$ with exponent $s$ supported on $E$, for some $s \in ((d+1)/2, \hdim(E))$.
		
		We can reduce the problem to when $E \subset [c,1]^d$ for a fixed constant $c$.
		To see this, 
		cut $\R^d$ into dyadic annuli $\{2^j \leq |x| \leq 2^{j+1}\}$. $\mu$ is positive on at least one of
		 these, and by rescaling we can assume it is $\{1/2 \leq |x| \leq 1\}$. If we cover $\{1/2 \leq |x| \leq
		 1\}$ with balls of radius $1/100$, $\mu$ is again positive on at least one of these. Notice that for
		 $\theta \in \SO(2)$, $(x\theta) \cdot (y\theta) = x \cdot y$. Thus rotating the measure $\mu$
		 does not affect the quantity in Theorem \ref{t:upper_bound} or \ref{t:lower_bound}, so we can
		 assume this ball is contained in $[c,1]^d$ for a fixed constant
		 $c$. Then we simply restrict $\mu$ to this ball and renormalize.

	\section{Proof of Theorem \ref{t:upper_bound} and Corollaries \ref{c:upper} and \ref{c:regular}}
		\label{sec:upper}
		\subsection{Proof of Theorem \ref{t:upper_bound}}
			We define a quantity $\cV_{T,t}$ which is approximately the quantity in Theorem \ref{t:upper_bound}. 
			Let $\rho$ be a smooth bump function on $\R$ supported around 0 and set $\rho^\epsilon(\cdot) = \epsilon^{-1} \rho(\epsilon^{-1} \cdot)$. Then define
			\begin{equation*}
				\cV^\epsilon_{T,t}(\mu) = \int \cdots \int \left( \prod_{(i,j) \in \mathcal E(T)} \rho^{\epsilon}(x_i\cdot x_j - t^{ij})\right) d\mu(x_1)\cdots d\mu(x_{k+1}).
			\end{equation*}
			The idea is to rip a leaf edge from $T$ one at a time until the tree is empty. One needs a corresponding mechanism that operators on $\cV_{T,t}^\epsilon(\mu)$ executing this plan, which is what we develop below. This is in the same spirit as M. Bennet, A. Iosevich, and K. Taylor's work on chain configurations \cite{chains}. 
			More concretely, we need to show that $\cV_{T,t}^\epsilon(\mu) \leq C$ independently of $\epsilon$ and $t$. We recast this problem in terms of operators for which we have nice results. 
			Define $\cR_t^\epsilon$ to be the operator with kernel $\rho^\epsilon(x \cdot y - t)$, that is 
			\begin{equation*}
				\cR^\epsilon_t f(x) = \int f(y) \rho^\epsilon(x \cdot y - t) dy.  
			\end{equation*}
			We also define 
			\begin{equation*}
				\cR^\epsilon_t (f\mu) (x) = \int f(y) \rho^\epsilon(x \cdot y - t) d\mu(y).
			\end{equation*}
			\footnote{This operator is known as the \emph{Radon Transform}. See Section \ref{subsec:radon} for more details or \cite{steinshakfunctional} for an in-depth review.}
			Then we can cast $\cV^\epsilon_{T,t}(\mu)$ in a way conducive to `ripping off' edges.
			\begin{definition}\label{d:inductive_def_upper}
				For $T$ a tree with a single vertex, define $f_T^\epsilon = 1$. Let $T$ be a tree with $k \geq 1$
				edges and say $y$ is a leaf with edge $(x,y)$. Say $(x,y)$ has corresponding dot product $t'$. Remove the leaf edge from $T$ to obtain a subtree $T'$. Define
				\begin{equation*}
					f_T^\epsilon = \cR_t^\epsilon(f_{T'}^\epsilon \mu).
				\end{equation*}
			\end{definition}
			One can think of $f_T^\epsilon(x)$ as $T$ pinned at $x$. Then integrating over the pinned point gives the entirety of $\cV^\epsilon_{T,t}(\mu)$. That is
			\begin{equation*}
				\cV^\epsilon_{T,t}(\mu) = \int f_T^\epsilon(x) d\mu(x).
			\end{equation*}
			By Cauchy-Schwarz and as $\mu$ is a probability measure,
			\begin{equation*}
				\cV^\epsilon_{T,t}(\mu) = \norm{f_T^\epsilon}_{L^1(\mu)} \leq \norm{f_T^\epsilon}_{L^2(\mu)}.
			\end{equation*}
			We use the following operator norm to run the induction.  
			\begin{theorem}\label{t:L2_mu_radon}
				If $\mu$ is a Frostman measure with exponent $s > (d+1)/2$ and with support as was established in Section \ref{sec:reductions}, $\cR_t^\epsilon$ is a bounded linear operator $L^2(\mu) \mapsto L^2(\mu)$ with
				\begin{equation*}
					\norma{\cR_t^\epsilon(f\mu)}_{L^2(\mu)} \lesssim \norm{f}_{L^2(\mu)}
				\end{equation*}
				independently of $\epsilon > 0$ and for $t \approx 1$.
			\end{theorem}
			The proof is left to Section \ref{subsec:radon}. 
			By our initial reduction $\mu$ has support in $[c,1]^d$, so $x \cdot y \approx 1$ on $\supp \mu$. In light of Definition \ref{d:inductive_def_upper}, $f_{T}^\epsilon =\cR_{t'}^\epsilon(f_{T'}^\epsilon)$.
			By Theorem \ref{t:L2_mu_radon} and the inductive hypothesis,
			\begin{align*}
				\norma
				{f_T^\epsilon}_{L^2(\mu)} &= \norma{\cR_{t'}^\epsilon(f_{T'}^\epsilon)}_{L^2(\mu)} \\
				&\lesssim \norma{f_{T'}^\epsilon}_{L^2(\mu)} \\
				&\lesssim 1
			\end{align*}
			independently of $\epsilon>0$ and $t = (t^{ij})$. \qed
		
		\subsection{Proof of Corollary \ref{c:upper}} 
			Consider any cover of $\Lambda_T(E)$ by products of intervals
			\begin{equation*}
				\Lambda_T(E) \subset
				\bigcup_\ell \prod_{(i,j) \in \E(T)} (t_\ell^{ij} - \epsilon_\ell, t_\ell^{ij} + \epsilon_\ell).
			\end{equation*}
			We have 
			\begin{align*}
				E^{k+1} &=\bigcup_{t \in \Lambda_T(E)} \{(x_1,\ldots, x_{k+1}) \in E^{k + 1} : x_i \cdot x_j = t^{ij}, (i,j) \in \E(T)\} \\
				&\subset 
				\bigcup_\ell \{(x_1,\ldots, x_{k+1}) \in E^{k + 1} : t_\ell^{ij} - \epsilon_\ell <x_i  \cdot x_j < t_\ell^{ij} + \epsilon_\ell, (i,j) \in \E(T)\},
			\end{align*}
			so by Theorem \ref{t:upper_bound},
			\begin{align*}
				1 &= \mu^{k+1}(E^{k+1})\\
				&\leq \sum_\ell \mu^{k+1}(\{(x_1,\ldots, x_{k+1}) \in E^{k + 1} : t_\ell^{ij} - \epsilon_\ell <x_i  \cdot x_j < t_\ell^{ij} + \epsilon_\ell, (i,j \in \E(T))\}) \\
				&< \sum_\ell C \epsilon_\ell.
			\end{align*}
			Thus $\sum_\ell \epsilon_\ell > 1/C$. This holds for any choice of covering so $|\Lambda_T(E)| \geq 1/C > 0$. \qed
			
		\subsection{Proof of Corollary \ref{c:regular}}
			In Definition $\ref{d:minkowski}$, we can replace $N(A, \epsilon)$ with $P(A, \epsilon)$. Here $P(A,\epsilon)$ is the packing number, the greatest number of disjoint $\epsilon$-balls with centers in $A$. This follows from the inequality 
			\begin{equation*}
				N(A,2\epsilon) \leq P(A,\epsilon) \leq N(A,\epsilon/2),
			\end{equation*}
			which one can find in a wonderful book by P. Mattila \cite{mattila_geometry}. Consider such a packing $\{B(x_i, \epsilon)\}$ of $T_t(E)$ of size $P(T_t(E), \epsilon)$. Since the centers of 
			the balls $B(x_i,\epsilon)$ are in $T_t(E)$, 
			\begin{equation*}
				\bigcup_i B(x_i, \epsilon) \subset (T_t(E))^{\epsilon}.
			\end{equation*}
			\footnote{$A^\epsilon$ is the $\epsilon$-neighborhood of $A$ defined as 
				$A^\epsilon = \{x \in \R^d : \exists y\in A, |x - y|<\epsilon\}$.}
			For any $(x_1,\ldots, x_{k+1}) \in E^{k+1} \cap (T_t(E))^{\epsilon}$,
			there are $x_1',\ldots, x_{k+1}' \in E$ such that $x_i' \cdot x_j' = t$ for $(i,j) \in E(T)$ and $|x_i - x_i'| < \epsilon$. 
			Thus
			\begin{align*}
				|x_i \cdot x_j - t| &\leq |x_i||x_j - x_j'| + |x_j'||x_i - x_i'| \\
				&\leq 2\epsilon,
			\end{align*}
			giving
			\begin{equation*}
				E^{k + 1} \cap (T_t(E))^{\epsilon} \subset 
				\{(x_1,\ldots, x_{k+1}) \in E^{k+1}: t^{ij} - 2\epsilon < x_i \cdot x_j < t^{ij} + 2\epsilon, (i,j) \in \E(T) \}.
			\end{equation*}
			We can conclude with Theorem \ref{t:upper_bound} that 
			\begin{align*}
				\sum_i \mu^{k+1}(B(x_i, \epsilon)) &=  \mu^{k + 1}\left(\bigcup_i B(x_i,\epsilon)\right) \\
				&\leq \mu^{k+1}(E^{k+1} \cap (T_t(E))^\epsilon) \\
				&\leq \mu^{k+1}(\{(x_1,\ldots, x_{k+1}) \in E^{k+1}: t^{ij} - 2\epsilon < x_i \cdot x_j < t^{ij} + 2\epsilon, (i,j) \in \E(T) \}) \\
				&< C\epsilon^k.
			\end{align*}
			Since $\mu(B(x,r)) \geq C'^{-1}r^s$, we get $\mu^{k+1}(B(x, r)) \geq C'^{-1}r^{(k+1)s}$. We conclude from the above calculation that 
			\begin{equation*}
				C'^{-1}\epsilon^{(k+1)s} P(T_t(E), \epsilon) < C\epsilon^k,
			\end{equation*} 
			so
			\begin{equation*}
				P(T_t(E), \epsilon) < C''\epsilon^{k-s}.
			\end{equation*}
			We obtain
			\begin{align*}
				\overline{\mdim}(T_t(E)) &= \limsup_{\epsilon \to 0} \frac{\log (P(T_t(E),
					 \epsilon))}{\log(1/\epsilon)} \\
				&\leq (k+ 1)s - k.
			\end{align*}
			\qed

	\section{Proof of Theorem \ref{t:lower_bound} and Corollary \ref{c:equal_embeddings}} \label{sec:lower_bound}
		In this section we consider $\cV_{T,t}$ with $t$ a scalar as
		\begin{equation*}
			\cV^\epsilon_{T,t}(\mu) = \int \cdots \int \left( \prod_{(i,j) \in \mathcal E(T)} \rho^{\epsilon}(x_i\cdot x_j - t)\right) d\mu(x_1)\cdots d\mu(x_{k+1}).
		\end{equation*}
		Our lower bound comes from repeated use of Holder's inequality, which `folds' the graph onto itself until reaching a single edge. $T$ itself is not guaranteed to enjoy enough symmetry for such an argument to work out, so we embed $T$ in a larger graph $\sigma(T)$ which is highly symmetric.

		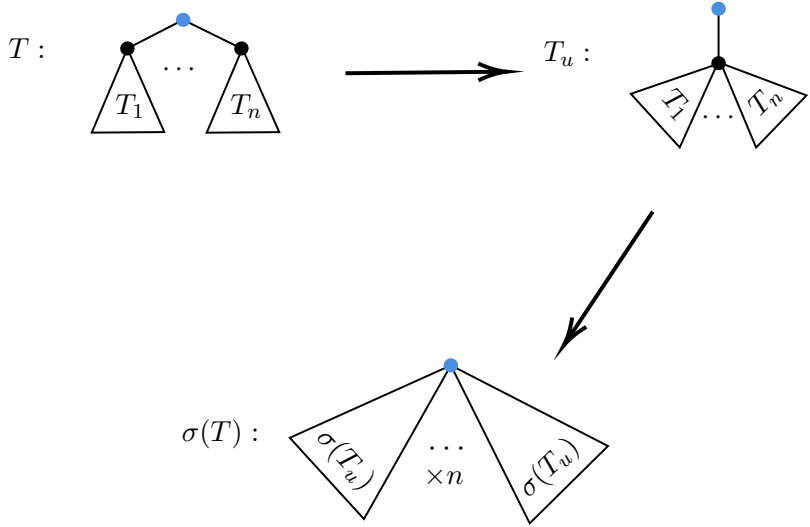
\begin{figure}
			\tikzset{every picture/.style={line width=0.75pt}} 
			
			\begin{tikzpicture}[x=0.75pt,y=0.75pt,yscale=-1,xscale=1]
				
				\draw   (278.15,274.94) -- (357.39,315.67) -- (317.91,354.67) -- cycle ;
				\draw [line width=1.5]    (225,127.5) -- (303.21,126.75) ;
				\draw [shift={(306.21,126.72)}, rotate = 179.45] [color={rgb, 255:red, 0; green, 0; blue, 0 }  ][line width=1.5]    (14.21,-4.28) .. controls (9.04,-1.82) and (4.3,-0.39) .. (0,0) .. controls (4.3,0.39) and (9.04,1.82) .. (14.21,4.28)   ;
				\draw [color={rgb, 255:red, 0; green, 0; blue, 0 }  ,draw opacity=1 ][fill={rgb, 255:red, 0; green, 0; blue, 0 }  ,fill opacity=1 ]   (115,115) -- (143.15,100.55) ;
				\draw [color={rgb, 255:red, 0; green, 0; blue, 0 }  ,draw opacity=1 ][fill={rgb, 255:red, 0; green, 0; blue, 0 }  ,fill opacity=1 ]   (172.5,115) -- (143.15,100.55) ;
				\draw  [color={rgb, 255:red, 74; green, 144; blue, 226 }  ,draw opacity=1 ][fill={rgb, 255:red, 74; green, 144; blue, 226 }  ,fill opacity=1 ] (140,100.55) .. controls (140,98.86) and (141.41,97.5) .. (143.15,97.5) .. controls (144.89,97.5) and (146.3,98.86) .. (146.3,100.55) .. controls (146.3,102.23) and (144.89,103.59) .. (143.15,103.59) .. controls (141.41,103.59) and (140,102.23) .. (140,100.55) -- cycle ;
				\draw   (115,115) -- (133.6,157.2) -- (97.02,157.47) -- cycle ;
				\draw  [color={rgb, 255:red, 0; green, 0; blue, 0 }  ,draw opacity=1 ][fill={rgb, 255:red, 0; green, 0; blue, 0 }  ,fill opacity=1 ] (169.74,116.53) .. controls (168.93,115.06) and (169.5,113.18) .. (171.02,112.34) .. controls (172.55,111.49) and (174.44,112) .. (175.26,113.47) .. controls (176.07,114.94) and (175.5,116.82) .. (173.98,117.66) .. controls (172.45,118.51) and (170.56,118) .. (169.74,116.53) -- cycle ;
				\draw  [color={rgb, 255:red, 0; green, 0; blue, 0 }  ,draw opacity=1 ][fill={rgb, 255:red, 0; green, 0; blue, 0 }  ,fill opacity=1 ] (112.24,116.53) .. controls (111.43,115.06) and (112,113.18) .. (113.52,112.34) .. controls (115.05,111.49) and (116.94,112) .. (117.76,113.47) .. controls (118.57,114.94) and (118,116.82) .. (116.48,117.66) .. controls (114.95,118.51) and (113.06,118) .. (112.24,116.53) -- cycle ;
				\draw   (172.98,115.03) -- (191.59,157.23) -- (155,157.5) -- cycle ;
				\draw [color={rgb, 255:red, 0; green, 0; blue, 0 }  ,draw opacity=1 ][fill={rgb, 255:red, 0; green, 0; blue, 0 }  ,fill opacity=1 ]   (413.15,122.44) -- (413.15,94.94) ;
				\draw  [color={rgb, 255:red, 74; green, 144; blue, 226 }  ,draw opacity=1 ][fill={rgb, 255:red, 74; green, 144; blue, 226 }  ,fill opacity=1 ] (410,94.94) .. controls (410,93.26) and (411.41,91.89) .. (413.15,91.89) .. controls (414.89,91.89) and (416.3,93.26) .. (416.3,94.94) .. controls (416.3,96.62) and (414.89,97.98) .. (413.15,97.98) .. controls (411.41,97.98) and (410,96.62) .. (410,94.94) -- cycle ;
				\draw   (412.5,122.89) -- (393.56,164.94) -- (368.9,137.91) -- cycle ;
				\draw  [color={rgb, 255:red, 0; green, 0; blue, 0 }  ,draw opacity=1 ][fill={rgb, 255:red, 0; green, 0; blue, 0 }  ,fill opacity=1 ] (410,122.44) .. controls (410,120.76) and (411.41,119.39) .. (413.15,119.39) .. controls (414.89,119.39) and (416.3,120.76) .. (416.3,122.44) .. controls (416.3,124.12) and (414.89,125.48) .. (413.15,125.48) .. controls (411.41,125.48) and (410,124.12) .. (410,122.44) -- cycle ;
				\draw   (414.33,122.44) -- (457.5,138.67) -- (432.09,165) -- cycle ;
				\draw [line width=1.5]    (379.97,197.37) -- (336.66,262.5) ;
				\draw [shift={(335,265)}, rotate = 303.62] [color={rgb, 255:red, 0; green, 0; blue, 0 }  ][line width=1.5]    (14.21,-4.28) .. controls (9.04,-1.82) and (4.3,-0.39) .. (0,0) .. controls (4.3,0.39) and (9.04,1.82) .. (14.21,4.28)   ;
				\draw   (278.15,274.94) -- (234.31,352.5) -- (196.91,311.51) -- cycle ;
				\draw  [color={rgb, 255:red, 74; green, 144; blue, 226 }  ,draw opacity=1 ][fill={rgb, 255:red, 74; green, 144; blue, 226 }  ,fill opacity=1 ] (275,274.94) .. controls (275,273.26) and (276.41,271.89) .. (278.15,271.89) .. controls (279.89,271.89) and (281.3,273.26) .. (281.3,274.94) .. controls (281.3,276.62) and (279.89,277.98) .. (278.15,277.98) .. controls (276.41,277.98) and (275,276.62) .. (275,274.94) -- cycle ;
				
				\draw (131,120.9) node [anchor=north west][inner sep=0.75pt]    {$\cdots $};
				\draw (107.5,136.4) node [anchor=north west][inner sep=0.75pt]    {$T_{1}$};
				\draw (165.48,136.43) node [anchor=north west][inner sep=0.75pt]    {$T_{n}$};
				\draw (53.5,108.4) node [anchor=north west][inner sep=0.75pt]    {$T:$};
				\draw (403,145) node [anchor=north west][inner sep=0.75pt]    {$\cdots $};
				\draw (391.57,131.62) node [anchor=north west][inner sep=0.75pt]  [rotate=-48.04]  {$T_{1}$};
				\draw (424.74,142.45) node [anchor=north west][inner sep=0.75pt]  [rotate=-314.4]  {$T_{n}$};
				\draw (323.5,109.9) node [anchor=north west][inner sep=0.75pt]    {$T_{u} :$};
				\draw (265.0,311.9) node [anchor=north west][inner sep=0.75pt]  [font=\large]  {$\cdots $};
				\draw (218.3,302.08) node [anchor=north west][inner sep=0.75pt]  [rotate=-48.04]  {$\sigma ( T_{u})$};
				\draw (307.75,334) node [anchor=north west][inner sep=0.75pt]  [rotate=-315.77]  {$\sigma ( T_{u})$};
				\draw (263,325.0) node [anchor=north west][inner sep=0.75pt]    {${\displaystyle \times n}$};
				\draw (141,300.9) node [anchor=north west][inner sep=0.75pt]    {$\sigma ( T) :$};

			\end{tikzpicture}
			
			\caption{Inductive construction of $\sigma(T)$.}
			\label{fig:tree_cover}
		\end{figure}
	
		\begin{definition}[Symmetric Tree
			Covers]\label{d:symm_tree_cover}
			Let $T$ be a tree with at least $k \geq 1$ edges. We define the symmetric tree cover $\sigma(T)$ of $T$ as follows. 
			If $k=1$ then $\sigma(T) = T$.
			Otherwise let $u$ be a non-leaf vertex of $T$. Let $T_u$ be the tree obtained by collapsing every neighbor of $u$ to a single vertex and reattaching $u$ to this vertex. Finally join $\deg(u)$ copies of $\sigma(T_u)$ at $u$ and call the resulting tree $\sigma(T)$.
		\end{definition}
		A diagram of the induction is provided in Figure \ref{fig:tree_cover}. A concrete example is in Figure \ref{fig:tree_example}.
		
		\begin{figure}
			\tikzset{every picture/.style={line width=0.75pt}} 
			
			\begin{tikzpicture}[x=0.75pt,y=0.75pt,yscale=-1,xscale=1]
				
				\draw [color={rgb, 255:red, 208; green, 2; blue, 27 }  ,draw opacity=1 ][fill={rgb, 255:red, 208; green, 2; blue, 27 }  ,fill opacity=1 ]   (123.49,280.62) -- (143.86,313.98) ;
				\draw [color={rgb, 255:red, 208; green, 2; blue, 27 }  ,draw opacity=1 ][fill={rgb, 255:red, 208; green, 2; blue, 27 }  ,fill opacity=1 ]   (166.92,281.4) -- (143.88,312.98) ;
				\draw [color={rgb, 255:red, 0; green, 0; blue, 0 }  ,draw opacity=1 ][fill={rgb, 255:red, 0; green, 0; blue, 0 }  ,fill opacity=1 ]   (405.57,20.81) -- (424.88,54.79) ;
				\draw [color={rgb, 255:red, 0; green, 0; blue, 0 }  ,draw opacity=1 ][fill={rgb, 255:red, 0; green, 0; blue, 0 }  ,fill opacity=1 ]   (441.93,20.63) -- (423.13,54.89) ;
				\draw  [color={rgb, 255:red, 208; green, 2; blue, 27 }  ,draw opacity=1 ][fill={rgb, 255:red, 208; green, 2; blue, 27 }  ,fill opacity=1 ] (100.27,106.95) .. controls (100.27,105.27) and (101.68,103.91) .. (103.42,103.91) .. controls (105.16,103.91) and (106.57,105.27) .. (106.57,106.95) .. controls (106.57,108.64) and (105.16,110) .. (103.42,110) .. controls (101.68,110) and (100.27,108.64) .. (100.27,106.95) -- cycle ;
				\draw  [color={rgb, 255:red, 208; green, 2; blue, 27 }  ,draw opacity=1 ][fill={rgb, 255:red, 208; green, 2; blue, 27 }  ,fill opacity=1 ] (143.05,136.69) .. controls (144.03,135.33) and (145.97,135.04) .. (147.39,136.05) .. controls (148.8,137.06) and (149.16,138.99) .. (148.18,140.36) .. controls (147.2,141.73) and (145.26,142.02) .. (143.84,141) .. controls (142.43,139.99) and (142.08,138.06) .. (143.05,136.69) -- cycle ;
				\draw [color={rgb, 255:red, 208; green, 2; blue, 27 }  ,draw opacity=1 ][fill={rgb, 255:red, 208; green, 2; blue, 27 }  ,fill opacity=1 ]   (103.42,106.95) -- (142.5,106.41) ;
				\draw [color={rgb, 255:red, 208; green, 2; blue, 27 }  ,draw opacity=1 ]   (144.84,105) -- (145.26,139.02) ;
				\draw [line width=1.5]    (145,176.29) -- (145,254.5) ;
				\draw [shift={(145,257.5)}, rotate = 270] [color={rgb, 255:red, 0; green, 0; blue, 0 }  ][line width=1.5]    (14.21,-4.28) .. controls (9.04,-1.82) and (4.3,-0.39) .. (0,0) .. controls (4.3,0.39) and (9.04,1.82) .. (14.21,4.28)   ;
				\draw  [color={rgb, 255:red, 208; green, 2; blue, 27 }  ,draw opacity=1 ][fill={rgb, 255:red, 208; green, 2; blue, 27 }  ,fill opacity=1 ] (145.09,69.35) .. controls (146.77,69.4) and (148.09,70.85) .. (148.04,72.59) .. controls (148,74.32) and (146.59,75.7) .. (144.91,75.65) .. controls (143.23,75.6) and (141.91,74.15) .. (141.96,72.41) .. controls (142,70.68) and (143.41,69.3) .. (145.09,69.35) -- cycle ;
				\draw [color={rgb, 255:red, 208; green, 2; blue, 27 }  ,draw opacity=1 ][fill={rgb, 255:red, 208; green, 2; blue, 27 }  ,fill opacity=1 ]   (145,72.5) -- (144.84,105) ;
				\draw  [color={rgb, 255:red, 208; green, 2; blue, 27 }  ,draw opacity=1 ][fill={rgb, 255:red, 208; green, 2; blue, 27 }  ,fill opacity=1 ] (113.15,78.49) .. controls (114.27,77.24) and (116.23,77.16) .. (117.53,78.32) .. controls (118.83,79.48) and (118.97,81.44) .. (117.85,82.69) .. controls (116.73,83.95) and (114.77,84.02) .. (113.47,82.86) .. controls (112.17,81.7) and (112.03,79.75) .. (113.15,78.49) -- cycle ;
				\draw [color={rgb, 255:red, 208; green, 2; blue, 27 }  ,draw opacity=1 ][fill={rgb, 255:red, 208; green, 2; blue, 27 }  ,fill opacity=1 ]   (115.5,80.59) -- (145,106.23) ;
				\draw  [color={rgb, 255:red, 208; green, 2; blue, 27 }  ,draw opacity=1 ][fill={rgb, 255:red, 208; green, 2; blue, 27 }  ,fill opacity=1 ] (187.5,105.61) .. controls (187.47,107.29) and (186.03,108.62) .. (184.29,108.59) .. controls (182.55,108.56) and (181.17,107.17) .. (181.2,105.48) .. controls (181.23,103.8) and (182.67,102.47) .. (184.41,102.5) .. controls (186.15,102.53) and (187.53,103.93) .. (187.5,105.61) -- cycle ;
				\draw [color={rgb, 255:red, 208; green, 2; blue, 27 }  ,draw opacity=1 ][fill={rgb, 255:red, 208; green, 2; blue, 27 }  ,fill opacity=1 ]   (184.35,105.55) -- (145.27,105.32) ;
				\draw  [color={rgb, 255:red, 208; green, 2; blue, 27 }  ,draw opacity=1 ][fill={rgb, 255:red, 208; green, 2; blue, 27 }  ,fill opacity=1 ] (219.56,82.38) .. controls (220.44,83.81) and (219.96,85.71) .. (218.48,86.62) .. controls (217,87.53) and (215.08,87.11) .. (214.2,85.68) .. controls (213.31,84.25) and (213.8,82.35) .. (215.28,81.44) .. controls (216.76,80.52) and (218.68,80.94) .. (219.56,82.38) -- cycle ;
				\draw [color={rgb, 255:red, 208; green, 2; blue, 27 }  ,draw opacity=1 ][fill={rgb, 255:red, 208; green, 2; blue, 27 }  ,fill opacity=1 ]   (216.88,84.03) -- (183.9,105) ;
				\draw  [color={rgb, 255:red, 208; green, 2; blue, 27 }  ,draw opacity=1 ][fill={rgb, 255:red, 208; green, 2; blue, 27 }  ,fill opacity=1 ] (219.44,129.31) .. controls (218.45,130.67) and (216.51,130.95) .. (215.1,129.93) .. controls (213.69,128.91) and (213.35,126.98) .. (214.33,125.62) .. controls (215.32,124.25) and (217.26,123.98) .. (218.67,124.99) .. controls (220.08,126.01) and (220.42,127.95) .. (219.44,129.31) -- cycle ;
				\draw [color={rgb, 255:red, 208; green, 2; blue, 27 }  ,draw opacity=1 ][fill={rgb, 255:red, 208; green, 2; blue, 27 }  ,fill opacity=1 ]   (216.89,127.46) -- (184.9,105) ;
				\draw  [color={rgb, 255:red, 74; green, 144; blue, 226 }  ,draw opacity=1 ][fill={rgb, 255:red, 74; green, 144; blue, 226 }  ,fill opacity=1 ] (142.16,105.68) .. controls (142.16,104) and (143.57,102.64) .. (145.31,102.64) .. controls (147.05,102.64) and (148.46,104) .. (148.46,105.68) .. controls (148.46,107.36) and (147.05,108.73) .. (145.31,108.73) .. controls (143.57,108.73) and (142.16,107.36) .. (142.16,105.68) -- cycle ;
				\draw [line width=1.5]    (247.5,105) -- (325.71,104.25) ;
				\draw [shift={(328.71,104.22)}, rotate = 179.45] [color={rgb, 255:red, 0; green, 0; blue, 0 }  ][line width=1.5]    (14.21,-4.28) .. controls (9.04,-1.82) and (4.3,-0.39) .. (0,0) .. controls (4.3,0.39) and (9.04,1.82) .. (14.21,4.28)   ;
				\draw [color={rgb, 255:red, 208; green, 2; blue, 27 }  ,draw opacity=1 ][fill={rgb, 255:red, 208; green, 2; blue, 27 }  ,fill opacity=1 ]   (383.15,92.83) -- (422.23,92.28) ;
				\draw  [color={rgb, 255:red, 0; green, 0; blue, 0 }  ,draw opacity=1 ][fill={rgb, 255:red, 0; green, 0; blue, 0 }  ,fill opacity=1 ] (347.87,113.44) .. controls (347.06,111.96) and (347.63,110.09) .. (349.15,109.24) .. controls (350.67,108.4) and (352.57,108.91) .. (353.38,110.38) .. controls (354.2,111.85) and (353.62,113.73) .. (352.1,114.57) .. controls (350.58,115.42) and (348.69,114.91) .. (347.87,113.44) -- cycle ;
				\draw [color={rgb, 255:red, 0; green, 0; blue, 0 }  ,draw opacity=1 ][fill={rgb, 255:red, 0; green, 0; blue, 0 }  ,fill opacity=1 ]   (350.63,111.91) -- (384.54,92.48) ;
				\draw  [color={rgb, 255:red, 0; green, 0; blue, 0 }  ,draw opacity=1 ][fill={rgb, 255:red, 0; green, 0; blue, 0 }  ,fill opacity=1 ] (345.24,76.8) .. controls (345.91,75.25) and (347.74,74.56) .. (349.34,75.24) .. controls (350.94,75.93) and (351.69,77.74) .. (351.03,79.29) .. controls (350.36,80.83) and (348.53,81.53) .. (346.93,80.84) .. controls (345.33,80.15) and (344.58,78.34) .. (345.24,76.8) -- cycle ;
				\draw [color={rgb, 255:red, 0; green, 0; blue, 0 }  ,draw opacity=1 ][fill={rgb, 255:red, 0; green, 0; blue, 0 }  ,fill opacity=1 ]   (347,77.48) -- (383.12,92.42) ;
				\draw [color={rgb, 255:red, 208; green, 2; blue, 27 }  ,draw opacity=1 ][fill={rgb, 255:red, 208; green, 2; blue, 27 }  ,fill opacity=1 ]   (393.18,69.96) -- (425.12,92.48) ;
				\draw  [color={rgb, 255:red, 0; green, 0; blue, 0 }  ,draw opacity=1 ][fill={rgb, 255:red, 0; green, 0; blue, 0 }  ,fill opacity=1 ] (352.6,59.72) .. controls (353.03,58.1) and (354.75,57.14) .. (356.43,57.59) .. controls (358.11,58.04) and (359.12,59.72) .. (358.69,61.34) .. controls (358.26,62.97) and (356.54,63.92) .. (354.86,63.48) .. controls (353.18,63.03) and (352.17,61.35) .. (352.6,59.72) -- cycle ;
				\draw [color={rgb, 255:red, 0; green, 0; blue, 0 }  ,draw opacity=1 ][fill={rgb, 255:red, 0; green, 0; blue, 0 }  ,fill opacity=1 ]   (355.64,60.53) -- (391.13,69.45) -- (393.55,70.05) ;
				\draw  [color={rgb, 255:red, 0; green, 0; blue, 0 }  ,draw opacity=1 ][fill={rgb, 255:red, 0; green, 0; blue, 0 }  ,fill opacity=1 ] (371.92,34.71) .. controls (373.37,33.85) and (375.26,34.37) .. (376.15,35.86) .. controls (377.04,37.36) and (376.59,39.27) .. (375.14,40.12) .. controls (373.7,40.98) and (371.81,40.47) .. (370.92,38.97) .. controls (370.03,37.48) and (370.48,35.57) .. (371.92,34.71) -- cycle ;
				\draw [color={rgb, 255:red, 0; green, 0; blue, 0 }  ,draw opacity=1 ][fill={rgb, 255:red, 0; green, 0; blue, 0 }  ,fill opacity=1 ]   (373.53,37.42) -- (393.98,70.73) ;
				\draw  [color={rgb, 255:red, 208; green, 2; blue, 27 }  ,draw opacity=1 ][fill={rgb, 255:red, 208; green, 2; blue, 27 }  ,fill opacity=1 ] (467.23,92.07) .. controls (467.23,93.75) and (465.81,95.11) .. (464.07,95.11) .. controls (462.33,95.1) and (460.93,93.74) .. (460.93,92.05) .. controls (460.93,90.37) and (462.35,89.01) .. (464.09,89.01) .. controls (465.83,89.02) and (467.23,90.38) .. (467.23,92.07) -- cycle ;
				\draw  [color={rgb, 255:red, 208; green, 2; blue, 27 }  ,draw opacity=1 ][fill={rgb, 255:red, 208; green, 2; blue, 27 }  ,fill opacity=1 ] (494.26,63.49) .. controls (495.37,64.76) and (495.21,66.71) .. (493.9,67.86) .. controls (492.6,69.01) and (490.64,68.91) .. (489.53,67.65) .. controls (488.42,66.38) and (488.58,64.43) .. (489.89,63.28) .. controls (491.19,62.13) and (493.15,62.23) .. (494.26,63.49) -- cycle ;
				\draw [color={rgb, 255:red, 208; green, 2; blue, 27 }  ,draw opacity=1 ][fill={rgb, 255:red, 208; green, 2; blue, 27 }  ,fill opacity=1 ]   (491.89,65.57) -- (462.88,91.76) ;
				\draw  [color={rgb, 255:red, 208; green, 2; blue, 27 }  ,draw opacity=1 ][fill={rgb, 255:red, 208; green, 2; blue, 27 }  ,fill opacity=1 ] (496.87,118.07) .. controls (495.83,119.4) and (493.88,119.6) .. (492.51,118.53) .. controls (491.14,117.46) and (490.87,115.51) .. (491.91,114.19) .. controls (492.94,112.87) and (494.9,112.66) .. (496.27,113.73) .. controls (497.64,114.8) and (497.91,116.75) .. (496.87,118.07) -- cycle ;
				\draw [color={rgb, 255:red, 208; green, 2; blue, 27 }  ,draw opacity=1 ][fill={rgb, 255:red, 208; green, 2; blue, 27 }  ,fill opacity=1 ]   (494.39,116.13) -- (463.27,92.48) ;
				\draw [color={rgb, 255:red, 208; green, 2; blue, 27 }  ,draw opacity=1 ][fill={rgb, 255:red, 208; green, 2; blue, 27 }  ,fill opacity=1 ]   (425.2,131.56) -- (424.49,92.48) ;
				\draw  [color={rgb, 255:red, 0; green, 0; blue, 0 }  ,draw opacity=1 ][fill={rgb, 255:red, 0; green, 0; blue, 0 }  ,fill opacity=1 ] (445.97,166.74) .. controls (444.5,167.57) and (442.62,167) .. (441.77,165.48) .. controls (440.92,163.97) and (441.42,162.07) .. (442.89,161.25) .. controls (444.36,160.43) and (446.24,160.99) .. (447.09,162.51) .. controls (447.94,164.03) and (447.44,165.92) .. (445.97,166.74) -- cycle ;
				\draw [color={rgb, 255:red, 0; green, 0; blue, 0 }  ,draw opacity=1 ][fill={rgb, 255:red, 0; green, 0; blue, 0 }  ,fill opacity=1 ]   (444.43,164) -- (424.86,130.17) ;
				\draw  [color={rgb, 255:red, 0; green, 0; blue, 0 }  ,draw opacity=1 ][fill={rgb, 255:red, 0; green, 0; blue, 0 }  ,fill opacity=1 ] (406.54,167.21) .. controls (405.07,166.39) and (404.56,164.5) .. (405.41,162.98) .. controls (406.25,161.46) and (408.13,160.89) .. (409.6,161.71) .. controls (411.07,162.53) and (411.58,164.42) .. (410.73,165.94) .. controls (409.88,167.46) and (408,168.03) .. (406.54,167.21) -- cycle ;
				\draw [color={rgb, 255:red, 0; green, 0; blue, 0 }  ,draw opacity=1 ][fill={rgb, 255:red, 0; green, 0; blue, 0 }  ,fill opacity=1 ]   (408.07,164.46) -- (426.61,130.06) ;
				\draw  [color={rgb, 255:red, 208; green, 2; blue, 27 }  ,draw opacity=1 ][fill={rgb, 255:red, 208; green, 2; blue, 27 }  ,fill opacity=1 ] (424.56,50.25) .. controls (426.24,50.25) and (427.6,51.67) .. (427.59,53.41) .. controls (427.59,55.15) and (426.22,56.55) .. (424.54,56.55) .. controls (422.85,56.54) and (421.49,55.13) .. (421.5,53.39) .. controls (421.51,51.65) and (422.87,50.24) .. (424.56,50.25) -- cycle ;
				\draw  [color={rgb, 255:red, 0; green, 0; blue, 0 }  ,draw opacity=1 ][fill={rgb, 255:red, 0; green, 0; blue, 0 }  ,fill opacity=1 ] (404.05,18.05) .. controls (405.53,17.24) and (407.4,17.82) .. (408.24,19.34) .. controls (409.08,20.87) and (408.56,22.76) .. (407.09,23.57) .. controls (405.61,24.38) and (403.74,23.8) .. (402.9,22.28) .. controls (402.06,20.75) and (402.58,18.86) .. (404.05,18.05) -- cycle ;
				\draw  [color={rgb, 255:red, 0; green, 0; blue, 0 }  ,draw opacity=1 ][fill={rgb, 255:red, 0; green, 0; blue, 0 }  ,fill opacity=1 ] (443.49,17.89) .. controls (444.95,18.72) and (445.44,20.61) .. (444.58,22.13) .. controls (443.72,23.64) and (441.84,24.2) .. (440.38,23.37) .. controls (438.92,22.54) and (438.43,20.64) .. (439.28,19.12) .. controls (440.14,17.61) and (442.02,17.06) .. (443.49,17.89) -- cycle ;
				\draw [color={rgb, 255:red, 208; green, 2; blue, 27 }  ,draw opacity=1 ][fill={rgb, 255:red, 208; green, 2; blue, 27 }  ,fill opacity=1 ]   (464.08,92.06) -- (425,92.52) ;
				\draw [color={rgb, 255:red, 208; green, 2; blue, 27 }  ,draw opacity=1 ][fill={rgb, 255:red, 208; green, 2; blue, 27 }  ,fill opacity=1 ]   (424.55,53.4) -- (424.96,92.48) ;
				\draw  [color={rgb, 255:red, 74; green, 144; blue, 226 }  ,draw opacity=1 ][fill={rgb, 255:red, 74; green, 144; blue, 226 }  ,fill opacity=1 ] (421.23,92.28) .. controls (421.23,90.6) and (422.64,89.23) .. (424.38,89.23) .. controls (426.12,89.23) and (427.53,90.6) .. (427.53,92.28) .. controls (427.53,93.96) and (426.12,95.33) .. (424.38,95.33) .. controls (422.64,95.33) and (421.23,93.96) .. (421.23,92.28) -- cycle ;
				\draw  [color={rgb, 255:red, 208; green, 2; blue, 27 }  ,draw opacity=1 ][fill={rgb, 255:red, 208; green, 2; blue, 27 }  ,fill opacity=1 ] (390.63,68.1) .. controls (391.62,66.74) and (393.56,66.47) .. (394.97,67.49) .. controls (396.38,68.51) and (396.72,70.45) .. (395.73,71.81) .. controls (394.74,73.17) and (392.8,73.44) .. (391.39,72.42) .. controls (389.98,71.4) and (389.64,69.47) .. (390.63,68.1) -- cycle ;
				\draw  [color={rgb, 255:red, 208; green, 2; blue, 27 }  ,draw opacity=1 ][fill={rgb, 255:red, 208; green, 2; blue, 27 }  ,fill opacity=1 ] (380,92.83) .. controls (380,91.14) and (381.41,89.78) .. (383.15,89.78) .. controls (384.89,89.78) and (386.3,91.14) .. (386.3,92.83) .. controls (386.3,94.51) and (384.89,95.87) .. (383.15,95.87) .. controls (381.41,95.87) and (380,94.51) .. (380,92.83) -- cycle ;
				\draw  [color={rgb, 255:red, 208; green, 2; blue, 27 }  ,draw opacity=1 ][fill={rgb, 255:red, 208; green, 2; blue, 27 }  ,fill opacity=1 ] (425.22,134.71) .. controls (423.54,134.71) and (422.17,133.31) .. (422.16,131.57) .. controls (422.15,129.83) and (423.51,128.41) .. (425.19,128.41) .. controls (426.87,128.4) and (428.24,129.8) .. (428.25,131.54) .. controls (428.26,133.28) and (426.9,134.7) .. (425.22,134.71) -- cycle ;
				\draw [color={rgb, 255:red, 208; green, 2; blue, 27 }  ,draw opacity=1 ]   (143.61,311.63) -- (144.03,345.65) ;
				\draw  [color={rgb, 255:red, 74; green, 144; blue, 226 }  ,draw opacity=1 ][fill={rgb, 255:red, 74; green, 144; blue, 226 }  ,fill opacity=1 ] (140.39,312.31) .. controls (140.39,310.63) and (141.8,309.27) .. (143.54,309.27) .. controls (145.28,309.27) and (146.69,310.63) .. (146.69,312.31) .. controls (146.69,313.99) and (145.28,315.36) .. (143.54,315.36) .. controls (141.8,315.36) and (140.39,313.99) .. (140.39,312.31) -- cycle ;
				\draw  [color={rgb, 255:red, 128; green, 128; blue, 128 }  ,draw opacity=1 ][fill={rgb, 255:red, 128; green, 128; blue, 128 }  ,fill opacity=1 ] (141.82,343.32) .. controls (142.8,341.96) and (144.74,341.67) .. (146.16,342.68) .. controls (147.57,343.69) and (147.92,345.62) .. (146.95,346.99) .. controls (145.97,348.36) and (144.03,348.65) .. (142.61,347.63) .. controls (141.2,346.62) and (140.84,344.69) .. (141.82,343.32) -- cycle ;
				\draw [line width=1.5]    (143.5,372.03) -- (143.5,450.24) ;
				\draw [shift={(143.5,453.24)}, rotate = 270] [color={rgb, 255:red, 0; green, 0; blue, 0 }  ][line width=1.5]    (14.21,-4.28) .. controls (9.04,-1.82) and (4.3,-0.39) .. (0,0) .. controls (4.3,0.39) and (9.04,1.82) .. (14.21,4.28)   ;
				\draw  [color={rgb, 255:red, 208; green, 2; blue, 27 }  ,draw opacity=1 ][fill={rgb, 255:red, 208; green, 2; blue, 27 }  ,fill opacity=1 ] (121.89,277.91) .. controls (123.33,277.06) and (125.23,277.58) .. (126.11,279.07) .. controls (127,280.57) and (126.54,282.48) .. (125.09,283.33) .. controls (123.65,284.19) and (121.75,283.67) .. (120.87,282.17) .. controls (119.98,280.68) and (120.44,278.77) .. (121.89,277.91) -- cycle ;
				\draw  [color={rgb, 255:red, 208; green, 2; blue, 27 }  ,draw opacity=1 ][fill={rgb, 255:red, 208; green, 2; blue, 27 }  ,fill opacity=1 ] (168.81,278.89) .. controls (170.15,279.9) and (170.4,281.84) .. (169.35,283.23) .. controls (168.31,284.63) and (166.37,284.93) .. (165.02,283.92) .. controls (163.68,282.91) and (163.44,280.97) .. (164.48,279.58) .. controls (165.53,278.18) and (167.46,277.88) .. (168.81,278.89) -- cycle ;
				\draw [color={rgb, 255:red, 208; green, 2; blue, 27 }  ,draw opacity=1 ]   (143.22,479.13) -- (143.64,513.15) ;
				\draw  [color={rgb, 255:red, 74; green, 144; blue, 226 }  ,draw opacity=1 ][fill={rgb, 255:red, 74; green, 144; blue, 226 }  ,fill opacity=1 ] (140,479.81) .. controls (140,478.13) and (141.41,476.77) .. (143.15,476.77) .. controls (144.89,476.77) and (146.3,478.13) .. (146.3,479.81) .. controls (146.3,481.49) and (144.89,482.86) .. (143.15,482.86) .. controls (141.41,482.86) and (140,481.49) .. (140,479.81) -- cycle ;
				\draw  [color={rgb, 255:red, 128; green, 128; blue, 128 }  ,draw opacity=1 ][fill={rgb, 255:red, 128; green, 128; blue, 128 }  ,fill opacity=1 ] (141.43,510.82) .. controls (142.41,509.46) and (144.35,509.17) .. (145.77,510.18) .. controls (147.18,511.19) and (147.53,513.12) .. (146.56,514.49) .. controls (145.58,515.86) and (143.64,516.15) .. (142.22,515.13) .. controls (140.81,514.12) and (140.45,512.19) .. (141.43,510.82) -- cycle ;
				\draw [line width=1.5]    (248.29,492.24) -- (326.5,491.49) ;
				\draw [shift={(329.5,491.46)}, rotate = 179.45] [color={rgb, 255:red, 0; green, 0; blue, 0 }  ][line width=1.5]    (14.21,-4.28) .. controls (9.04,-1.82) and (4.3,-0.39) .. (0,0) .. controls (4.3,0.39) and (9.04,1.82) .. (14.21,4.28)   ;
				\draw [line width=1.5]    (250.5,305.74) -- (328.71,304.99) ;
				\draw [shift={(331.71,304.96)}, rotate = 179.45] [color={rgb, 255:red, 0; green, 0; blue, 0 }  ][line width=1.5]    (14.21,-4.28) .. controls (9.04,-1.82) and (4.3,-0.39) .. (0,0) .. controls (4.3,0.39) and (9.04,1.82) .. (14.21,4.28)   ;
				\draw [color={rgb, 255:red, 208; green, 2; blue, 27 }  ,draw opacity=1 ]   (426.11,471.63) -- (426.53,505.65) ;
				\draw  [color={rgb, 255:red, 74; green, 144; blue, 226 }  ,draw opacity=1 ][fill={rgb, 255:red, 74; green, 144; blue, 226 }  ,fill opacity=1 ] (422.89,472.31) .. controls (422.89,470.63) and (424.3,469.27) .. (426.04,469.27) .. controls (427.78,469.27) and (429.19,470.63) .. (429.19,472.31) .. controls (429.19,473.99) and (427.78,475.36) .. (426.04,475.36) .. controls (424.3,475.36) and (422.89,473.99) .. (422.89,472.31) -- cycle ;
				\draw  [color={rgb, 255:red, 128; green, 128; blue, 128 }  ,draw opacity=1 ][fill={rgb, 255:red, 128; green, 128; blue, 128 }  ,fill opacity=1 ] (424.32,503.32) .. controls (425.3,501.96) and (427.24,501.67) .. (428.66,502.68) .. controls (430.07,503.69) and (430.42,505.62) .. (429.45,506.99) .. controls (428.47,508.36) and (426.53,508.65) .. (425.11,507.63) .. controls (423.7,506.62) and (423.34,504.69) .. (424.32,503.32) -- cycle ;
				\draw [color={rgb, 255:red, 208; green, 2; blue, 27 }  ,draw opacity=1 ][fill={rgb, 255:red, 208; green, 2; blue, 27 }  ,fill opacity=1 ]   (403.49,278.12) -- (423.86,311.48) ;
				\draw [color={rgb, 255:red, 208; green, 2; blue, 27 }  ,draw opacity=1 ][fill={rgb, 255:red, 208; green, 2; blue, 27 }  ,fill opacity=1 ]   (446.92,278.9) -- (423.88,310.48) ;
				\draw [color={rgb, 255:red, 208; green, 2; blue, 27 }  ,draw opacity=1 ]   (423.61,309.13) -- (424.03,343.15) ;
				\draw  [color={rgb, 255:red, 74; green, 144; blue, 226 }  ,draw opacity=1 ][fill={rgb, 255:red, 74; green, 144; blue, 226 }  ,fill opacity=1 ] (420.39,309.81) .. controls (420.39,308.13) and (421.8,306.77) .. (423.54,306.77) .. controls (425.28,306.77) and (426.69,308.13) .. (426.69,309.81) .. controls (426.69,311.49) and (425.28,312.86) .. (423.54,312.86) .. controls (421.8,312.86) and (420.39,311.49) .. (420.39,309.81) -- cycle ;
				\draw  [color={rgb, 255:red, 128; green, 128; blue, 128 }  ,draw opacity=1 ][fill={rgb, 255:red, 128; green, 128; blue, 128 }  ,fill opacity=1 ] (421.82,340.82) .. controls (422.8,339.46) and (424.74,339.17) .. (426.16,340.18) .. controls (427.57,341.19) and (427.92,343.12) .. (426.95,344.49) .. controls (425.97,345.86) and (424.03,346.15) .. (422.61,345.13) .. controls (421.2,344.12) and (420.84,342.19) .. (421.82,340.82) -- cycle ;
				\draw  [color={rgb, 255:red, 208; green, 2; blue, 27 }  ,draw opacity=1 ][fill={rgb, 255:red, 208; green, 2; blue, 27 }  ,fill opacity=1 ] (401.89,275.41) .. controls (403.33,274.56) and (405.23,275.08) .. (406.11,276.57) .. controls (407,278.07) and (406.54,279.98) .. (405.09,280.83) .. controls (403.65,281.69) and (401.75,281.17) .. (400.87,279.67) .. controls (399.98,278.18) and (400.44,276.27) .. (401.89,275.41) -- cycle ;
				\draw  [color={rgb, 255:red, 208; green, 2; blue, 27 }  ,draw opacity=1 ][fill={rgb, 255:red, 208; green, 2; blue, 27 }  ,fill opacity=1 ] (448.81,276.39) .. controls (450.15,277.4) and (450.4,279.34) .. (449.35,280.73) .. controls (448.31,282.13) and (446.37,282.43) .. (445.02,281.42) .. controls (443.68,280.41) and (443.44,278.47) .. (444.48,277.08) .. controls (445.53,275.68) and (447.46,275.38) .. (448.81,276.39) -- cycle ;
				\draw [line width=1.5]    (425,445) -- (424.22,366.79) ;
				\draw [shift={(424.19,363.79)}, rotate = 89.43] [color={rgb, 255:red, 0; green, 0; blue, 0 }  ][line width=1.5]    (14.21,-4.28) .. controls (9.04,-1.82) and (4.3,-0.39) .. (0,0) .. controls (4.3,0.39) and (9.04,1.82) .. (14.21,4.28)   ;
				\draw [line width=1.5]    (425.81,261.21) -- (425.03,183) ;
				\draw [shift={(425,180)}, rotate = 89.43] [color={rgb, 255:red, 0; green, 0; blue, 0 }  ][line width=1.5]    (14.21,-4.28) .. controls (9.04,-1.82) and (4.3,-0.39) .. (0,0) .. controls (4.3,0.39) and (9.04,1.82) .. (14.21,4.28)   ;
				
				\definecolor{Blue}{RGB}{0.29,0.56,0.89}
				\draw (110,201.19) node [anchor=north west][inner sep=0.75pt]    {$T_{\textcolor[rgb]{0.29,0.56,0.89}{\bullet}}$};
				\draw (283,79.9) node [anchor=north west][inner sep=0.75pt]    {$\sigma $};
				\draw (108.5,396.93) node [anchor=north west][inner sep=0.75pt]    {$T_{\textcolor[rgb]{0.29,0.56,0.89}{\bullet}}$};
				\draw (283.79,467.14) node [anchor=north west][inner sep=0.75pt]    {$\sigma $};
				\draw (286,280.64) node [anchor=north west][inner sep=0.75pt]    {$\sigma $};
				\draw (376,389.9) node [anchor=north west][inner sep=0.75pt]  [font=\large]  {$\bowtie \textcolor[rgb]{0.5,0.5,0.5}{_{\bullet}}$};
				\draw (371,207.4) node [anchor=north west][inner sep=0.75pt]  [font=\large]  {$\bowtie \textcolor[rgb]{0.5,0.5,0.5}{_{\bullet}}$};

			\end{tikzpicture}
			
			\caption{Unwrapping the induction for a concrete tree. Downward arrows are the construction of the $T_u$'s as in Definition \ref{d:symm_tree_cover}.
				Rightward arrows are the $\sigma$ operation. 
				Upward arrows are the joining operation as in the last step of Definition \ref{d:symm_tree_cover}.
				Edges in the upper left graph are tracked in red.}
			\label{fig:tree_example}
		\end{figure}
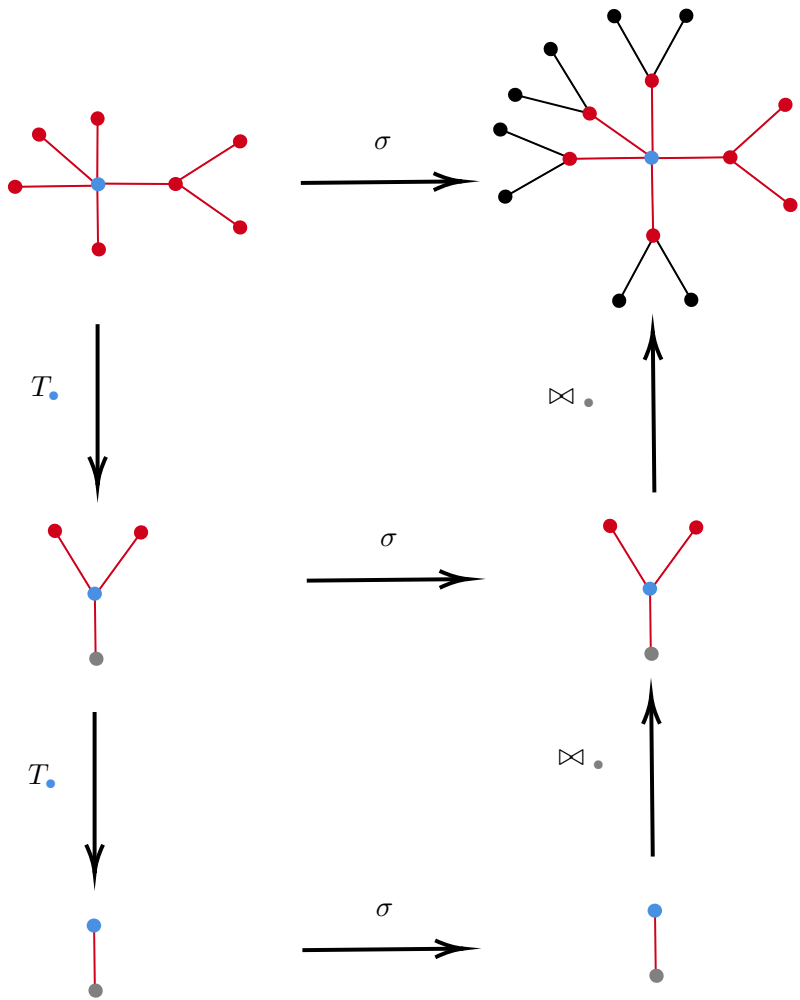

		Note that the symmetric tree covering can depend on the choice of pivot at each stage, but we have no need for uniqueness.
		It is not difficult to establish the following properties of $\sigma(T)$.
		\begin{lemma}
			$T \subset \sigma(T)$ and $\sigma(T)$ is finite.
		\end{lemma}
		\begin{proof}
			We proceed by induction. If $T$ is a single edge we are done. Suppose that $T$ has $k \geq 2$ edges and $T' \subset \sigma(T')$ for any tree $T'$ on fewer than $k$ edges. $T_u$ contains $k - \delta(u) + 1$ edges, and $\delta(u) > 1$ as we can take $u$ to be a non-leaf vertex. Thus $T_u \subset \sigma(T_u)$. $\sigma(T_u)$ contains a copy of $T_u$ for each vertex, so it contains each connected component of $T\setminus u$ connected to $u$, giving $T \subset \sigma(T)$. 
			Also $\sigma(T)$ is finite as $\sigma(T_u)$ is by induction finite. 
		\end{proof}
		
		The base case when $T$ has a single edge is contained in a paper by A. Greenleaf, A. Iosevich, and K. Taylor \cite{nonempty_interior_radon}. We give the theorem in our notation below, letting $e$ denote an edge.
		\begin{theorem}[Greenleaf--Iosevich--Taylor \cite{nonempty_interior_radon}]
			Let $\mu$ have exponent $s > (d+1)/2$. Then there is an open interval $I$ such that
			\begin{equation*}
				\cV^\epsilon_{e,t} = \ang{\cR_t^\epsilon \mu, \mu} \gtrsim 1
			\end{equation*}
			independently of $\epsilon > 0$ and $t \in I$. 
		\end{theorem}
		
		Consider when $T$ has $k \geq 2$ edges and let $T_u$ be as
		in Definition \ref{d:symm_tree_cover}. Note that the disjoint copies of $T_u$ are common only in $u$. We have
		\begin{equation*}
			\cV_{\sigma(T), t}^\epsilon(\mu) = \int 
			\left( \int \cdots \int \prod_{(i,j) \in \E(\sigma(T_u))}  \rho^\epsilon(x_i \cdot x_j - t) d\mu(x_1) \cdots d\mu(x_{k'}) \right)^{\deg(u)} d\mu(u),
		\end{equation*}
		where we are abusing notation and letting $u$ represent the vertex as embedded in $E$ and as in the abstract graph $T$. Here $x_1,\ldots, x_{k'}$ are the vertices in $T_u$ excluding $u$. Since $\mu$ is a probability measure, Holder gives
		\begin{align*}
			\int &\left(\int \cdots \int \prod_{(i,j) \in \E(\sigma(T_u))}  \rho^\epsilon(x_i \cdot x_j - t) d\mu(x_1) \cdots d\mu(x_{k'}) \right)^{\deg(u)} d\mu(u) \\
			&\geq \left(\int\int \cdots \int  \prod_{(i,j) \in \E(\sigma(T_u))} \rho^\epsilon(x_i \cdot x_j - t) d\mu(x_1) \cdots d\mu(x_{k'}) d\mu(u) \right)^{\deg(u)}\\
			&= (\cV_{\sigma(T_u), t}^\epsilon(\mu))^{\deg(u)} \gtrsim 1,
		\end{align*}
		where the last line is by induction. \qed
		
		\subsection{Proof of Corollary \ref{c:equal_embeddings}}
			Set
			\begin{equation*}
				K_n = \{(x_1,\ldots, x_{k+1}) \in E^{k+1} : t - 1/n \leq x_i \cdot x_j \leq t + 1/n, (i,j) \in \E(T)\}. 
			\end{equation*}
			Then the $K_n$ are nested, non-increasing, and
			\begin{equation*}
				\bigcap_{n \geq 1} K_n = \sigma(T)_t(E). 
			\end{equation*}
			Consider $\Phi(x_1,\ldots, x_{k+1}) = (x_i \cdot x_j)_{(i,j) \in E(T)}$,
			which is continuous as a function $\R^{k+1} \to \R^k$.
			Then $K_n = E^{k+1} \cap \Phi^{-1}([t-1/n, t+1/n])$ is compact, being the intersection of a compact set and a closed set. By Theorem \ref{t:lower_bound}
			$\mu(K_n) \geq c/n^k > 0$, so $K_n$ is nonempty. By Cantor's intersection theorem
			$\sigma(T)_t(E)$ is thus nonempty. 
			By the inclusion of $T$ in $\sigma(T)$, $T_t(E)$ is nonempty. 
			\qed

	\section{Appendix}
	
		\subsection{The Radon Transform} \label{subsec:radon}
		As alluded to in Section \ref{sec:upper}, we have a family of Radon transforms $\cR_t^\epsilon$ given by 
		\begin{equation*}
			\cR_t^\epsilon f(x) = \int f(y) \rho^\epsilon(x \cdot y - t) dy.
		\end{equation*}
		When $\mu$ is a Borel measure we write 
		\begin{equation*}
			\cR_t^\epsilon(f\mu)(x) = \int f(y) \rho^\epsilon(x \cdot y - t) d\mu(y).
		\end{equation*}
		The following is a well-known mapping property of $\cR^\epsilon_t$. See \cite{PhongSteinRadon} for the original argument. 
		For a source with more background, see Stein and Shakarchi's book on functional analysis \cite{steinshakfunctional}. 
		\begin{theorem} \label{t:Radon_sobolev}
			$\cR_t^\epsilon$ is a bounded linear operator $L^2(\R^d) \to L^2_{(d-1)/2}(\R^d)$ with 
			\begin{equation*}
				\norma{\cR_t^\epsilon f}_{L_{(d-1)/2}^2(\R^d)} \lesssim \norma{f}_{L^2(\R^d)}
			\end{equation*}
			independently of $\epsilon > 0$ and for $t \approx 1$. 
		\end{theorem}
		Recall that the Sobolev space $L^2_\alpha(\R^d)$ is the function space
		equipped with the norm
		\begin{equation*}
			\norma{f}_{L_\alpha^2(\R^d)} := \left(\int (1 + |\xi|^2)^{\alpha} |\widehat{f}(\xi)|^2 d\xi \right)^{1/2}.
		\end{equation*}
		We bootstrap off this result to show $L^2(\mu) \to L^2(\mu)$ boundedness of $\cR_t^\epsilon$.
		This was done for convolution operators in \cite{iosevich2019maximal}, but the method also applies (as the authors  of \cite{iosevich2019maximal} remark) to Radon-type operators. We give the proof below.
		\begin{mythm}{\ref{t:L2_mu_radon}}
			If $\mu$ is a Frostman measure with exponent $s > (d+1)/2$ and with support as was established in Section \ref{sec:reductions}, $\cR_t^\epsilon$ is a bounded linear operator $L^2(\mu) \mapsto L^2(\mu)$ with
			\begin{equation*}
				\norma{\cR_t^\epsilon(f\mu)}_{L^2(\mu)} \lesssim \norma{f}_{L^2(\mu)}
			\end{equation*}
			independently of $\epsilon > 0$ and for $t \approx 1$.
		\end{mythm}
		\begin{proof}
			By polarization, it suffices to show that for 
			any $\norma{g}_{L^2(\nu)} \leq 1$,
			\begin{equation*}
				|\ang{\cR^\epsilon_t(f\mu), g\mu}| \lesssim \norma{f}_{L^2(\mu)} \norma{g}_{L^2(\mu)}
			\end{equation*}
			independently of $\epsilon > 0$ and $t \approx 1$.
			To proceed, we localize to dyadic frequencies. We will see that the large frequencies are the only ones that give us any trouble, so we consider a partition of unity 
			\begin{equation*}
				\sum_{j\geq 1} \chi(2^{-j}\xi) + \chi_0(\xi) = 1,
			\end{equation*}
			where $\chi$ is supported in the annulus $\{\xi : 1/4 \leq |\xi| \leq 1\}$. One can construct such functions by considering a $C^\infty$ function $\phi$ equal to 1 when $|\xi| \geq 1$ and to $0$ when $|x| \leq 1/2$, and letting $\chi(\xi) = \phi(2\xi) - \phi(\xi)$.
			The function $\chi_0$ is supported in the ball $\{\xi : |\xi| \leq 1\}$ and is also smooth, since $\sum_{j \geq 1} \chi(2^{-j} \cdot)$ being the sum of only a finite number of smooth functions in a neighborhood of any point is smooth. Now we can define the Littlewood-Paley projection by the relation 
			\begin{equation*}
				\widehat{P_j f}(\xi) = \widehat{f}(\xi) \chi(2^{-j} \xi)
			\end{equation*}
			for $j \geq 1$, and 
			\begin{equation*}
				\widehat{P_0 f}(\xi) = \widehat{f}(\xi) \chi_0(\xi).
			\end{equation*}
			Then $f = \sum_{j \geq 0} P_j f$.
			Applying the Littlewood-Paley decomposition to $f \mu$ and $g \mu$, we obtain
			\begin{align}
				|\ang{\cR^\epsilon_t(f\mu), g\mu}|& \leq \sum_{j, k \geq 0} |\ang{\cR^\epsilon_t(P_j(f\mu)),
			    P_k(g\mu)}|
				\\ 
				&= \sum_{|j-k| \leq M} |\ang{\cR^\epsilon_t(P_j(f\mu)), P_k(g\mu)}| + \sum_{|j-k| > M}
				|\ang{\cR^\epsilon_t(P_j(f\mu)), P_k(g\mu)}|,
				\label{eq:sum}
			\end{align}
		
			where $M$ is a constant to be chosen later. We handle the $|j-k| \leq M$ portion first.
			We need a mechanism to transfer an $L^2_{(d+1)/2}$ bound to a $L^2(\mu)$ one.
			We need the following generic test for $L^2$ boundedness. See \cite{wolff2003lectures} Lemma 7.5.
			\begin{theorem}[Schur's test]
				Let $(X, \mu)$ and $(Y,\nu)$ be measure spaces, and let $K(x,y)$ be a measurable function on $X \times Y$ with
				\begin{equation*}
					\int_X |K(x,y)|d\mu(x) \leq A \text{ for each $y$},
				\end{equation*}
				\begin{equation*}
					\int_Y |K(x,y)|d\nu(x) \leq B \text{ for each $x$}.
				\end{equation*}
				Define $T_K f(x) = \int K(x,y)f(y)d\nu(y)$. Then there is an estimate 
				\begin{equation*}
					\norma{T_K f}_{L^2(\mu)} \leq \sqrt{AB} \norma{f}_{L^2(\mu)}.
				\end{equation*} 
			\end{theorem}
			Now we can continue to prove the following.
			\begin{lemma}\label{l:wolff_lemma}
				If $\mu$ is a Frostman measure with exponent $s$ we have the estimate
				\begin{equation*}
					\norma{\widehat{f\mu}}_{L^2(|\xi| \leq 2^j)} \lesssim 2^{j(d-s)/2}\norma{f}_{L^2(\mu)}.
				\end{equation*}
			\end{lemma}
			Lemma \ref{l:wolff_lemma} is a special case of a result in \cite{StrichartzFourierAsymptotics}. 
			A straightforward proof of Lemma \ref{l:wolff_lemma} is contained in \cite{wolff2003lectures}, but we give it here with all the details. 
			\begin{proof}
				Let $\phi$ be an even Schwarz function which is $\geq 1$ on the unit ball and whose Fourier transform has compact support. It is not difficult to see that such a function exists. For example, let $f$ be a real-valued, nonnegative, symmetric $C_0^\infty(\R^d)$ function such as
				\begin{equation*}
					f(x) = \begin{cases}
						\exp(- \frac{1}{1-|x|^2}) &|x| \leq 1 \\
						0 &\text{otherwise}
					\end{cases}.
				\end{equation*}
				Define $g := \mathcal F^{-1}(f)$. Then 
				\begin{equation*}
					g(0) = \int f(x) dx > 0,
				\end{equation*}
				and by rescaling $f$ we can assume $g(0) \geq 2$. $g$ is certainly continuous, so for some $\delta > 0$, $g \geq 1$  on $\{|x| \leq \delta \}$.
				Finally define $\phi(x) = \phi(\delta x)$ which is as desired. 
				
				Now we define $\phi_j(\cdot) = \phi(2^{-j} \cdot)$, which is at least 1 on 
				$\{|\xi| \leq 2^j \}$. Using this and Plancherel, 
				\begin{align*}
					\norma{\widehat{f\mu}}_{L^2(|\xi| \leq 2^j)} &\leq \norma{\phi_j \widehat{f\mu}}_{L^2} \\
					&= \norma{\widehat{\phi_j} *(f\mu)}_{L^2}.
				\end{align*}
				This last line is the $L^2$ norm of the function 
				\begin{equation*}
					x\mapsto \int 2^{jd} \widehat{\phi}(2^j(x-y)) f(y) d\mu(y).
				\end{equation*}
				We have
				\begin{equation*}
					\int |2^{jd} \widehat{\phi}(2^j(x-y))| dx = \norma{\widehat{\phi}}_{L^1}
				\end{equation*}
				by a change of variables. $\widehat \phi$ has compact support in some fixed ball $M$, so 
				\begin{align*}
					\int |2^{jd} \widehat{\phi}(2^j(x-y))| d\mu(y) 
					&= 2^{jd} \int_{|x-y| \leq M2^{-j}}  |\widehat{\phi}(2^j(x-y))| d\mu(y) \\
					&\lesssim 2^{j(d-s)}.
				\end{align*}
				The last line follows from the fact that $\mu(B(x,r)) \lesssim r^s$.
				Now we can apply Schur's test with the kernel $K(x,y) = 2^{jd}\widehat{\phi}(2^j (x-y))$ and obtain 
				\begin{equation*}
					\norma{\widehat{\phi_j} *(f\mu)}_{L^2} \lesssim 2^{j(d-s)/2} \norma{f}_{L^2(\mu)}.
				\end{equation*}
			\end{proof}
			Returning to the proof of Theorem $\ref{t:L2_mu_radon}$,
			Plancherel and Cauchy-Schwarz give that the $|j-k| \leq M$ portion of \eqref{eq:sum} is dominated by 
			\begin{align*}
				\sum_{|j-k| \leq M} \norma{\widehat{\cR^\epsilon_t(P_j(f\mu))}}_{L^2(|\tau| \approx 2^k)} \norma{\widehat{P_k(g\mu)}}_{L^2}.
			\end{align*}
			We take advantage of Theorem \ref{t:Radon_sobolev}, the $L^2(\R^d) \to L^2_{(d-1)/2}(\R^d)$ boundedness of the Radon transform $\cR_t^\epsilon$.
			We have
			\begin{align*}
				\norma{\widehat{\cR^\epsilon_t(P_j(f\mu))}}_{L^2(|\tau| \approx 2^k)} &\lesssim 2^{-k(d-1)/2} 
				\left(\int (1+|\tau|^2)^{(d-1)/2} |\widehat{\cR^\epsilon_t(P_j(f\mu))}(\tau)|^2 d\tau \right)^{1/2}\\
				&= 2^{-k(d-1)/2} \norma{\cR^\epsilon_t(P_j(f\mu))}_{L^2_{(d-1)/2}} \\
				&\leq 2^{-k(d-1)/2} \norma{P_j(f\mu)}_{L^2} \\
				&\lesssim 2^{-k(d-1)/2} 2^{j(d-s)/2}.
			\end{align*}
			Thus we are left with
			\begin{align*}
				\sum_{|j-k| \leq M} \norma{\widehat{\cR^\epsilon_t(P_j(f\mu))}}_{L^2(|\tau| \approx 2^k)} \norma{\widehat{P_k(g\mu)}}_{L^2}
				\lesssim \sum_{|j-k| \leq M} 2^{-k(d-1)/2}2^{j(d-s)/2} 2^{k(d-s)/2},
			\end{align*}
			which is summable if $s > (d + 1) / 2$. 
			
			\bigskip
			We still need to handle the $|j-k| > M$ portion of \eqref{eq:sum}.
			This diagonalization can be executed with the following Lemma.
			\begin{lemma} \label{l:diag}
				For any positive integer $N$, there is a $K$ such that if $|j - k| > K$, 
				\begin{align*}
					|\ang{\cR^\epsilon_t(P_j(f\mu)), P_k(g\mu)}| \lesssim_N 2^{-N\max(j,k)} \norma{f}_{L^2(\mu)} 
					\norma{g}_{L^2(\mu)},
				\end{align*}
				independently of $\epsilon > 0$ and for $t \approx 1$. 
			\end{lemma}
			We give the proof below. With Lemma \ref{l:diag} we see that
			\begin{equation*}
				\sum_{|j-k| > M} |\ang{\cR_t^\epsilon(P_j(f\mu)), P_k(g\mu)}| \lesssim_N \norma{f}_{L^2(\mu)} \norma{g}_{L^2(\mu)}\sum_{|j-k| > M} 2^{-N\max(j,k)}.
			\end{equation*}
			This is summable even when $N = 1$. \qed

			\begin{proof}[Proof of Lemma \ref{l:diag}]
				We will argue by nonstationary phase. By the support properties of $\mu$, we can insert a bump function $\eta$ with support in $[c,1]^d$. By Fourier inversion on $P_j(f\mu)$ and $\rho^\epsilon$,
				\begin{align*}
					\cR_t^\epsilon (P_j(f\mu)) (x) &=
					\int P_j(f\mu)(y) \eta(x,y)  \rho^\epsilon(x \cdot y - t) dy \\
					&= \iiint e^{2\pi i(y \cdot \xi + s(x\cdot y - t))} \widehat{P_j (f\mu)}(\xi) \eta(x,y) \widehat{\rho}(\epsilon s) d\xi ds dy
				\end{align*}
				Taking the Fourier transform of $\cR_t(P_j(f\mu))$, 
				\begin{align*}
					\widehat{\cR_t(P_j(f\mu))}(\tau) &= \iiiint e^{2\pi i(y \cdot \xi -  x \cdot \tau+ s(x\cdot y - t))} \widehat{P_j (f\mu)}(\xi) \eta(x,y) \widehat{\rho}(\epsilon s) d\xi ds dy.
				\end{align*}
				Finally Plancherel gives
				\begin{align*}
					\ang{\cR_t^\epsilon(P_j(f\mu)), P_k(g\mu)} = 
					\iiint \widehat{P_j(f\mu)}(\xi)  \widehat{P_k(g\mu)}(\tau) \widehat{\rho}(\epsilon s) d\xi d\tau ds,
				\end{align*}
				where 
				\begin{equation*}
					I_{jk}(\xi, \tau, s) = \chi_j(\xi) \chi_k(\tau)\iint e^{2\pi i(y \cdot \xi - x\cdot \tau + s(x\cdot y - t))}\eta(x,y) dx dy.
				\end{equation*}
				Inserting the smooth cutoffs from the definition of the Littlewood-Paley decomposition is justified as $\chi_j \approx \chi_j^2$.
				For convenience we write 
				\begin{equation*}
					\Psi_{\xi,\tau,s}(x,y) = y\cdot \xi - x\cdot \tau + s(x\cdot y - t).
				\end{equation*}
				We would be done if we could show that
				\begin{equation} \label{eq:diag:meat}
					|I_{jk}(\xi,\tau, s)| \lesssim_N (1 + |s|)^{-2} 2^{-N\max(j,k)}.
				\end{equation}
				$\rho$ is Schwarz since it is even $C_0^\infty$, so $\widehat \rho$ is Schwarz. This gives $|\widehat{\rho}(\epsilon s)| \lesssim 1$.
				It would follow from $|\widehat{\rho}(\epsilon s| \lesssim 1$ and \eqref{eq:diag:meat} that  
				\begin{align*}
					|\ang{\cR_t^\epsilon(P_j(f\mu)), P_k(g\mu)}| &\leq \iiint |\widehat{P_j(f\mu)}(\xi)| |\widehat{P_k(g\mu)}(\tau)| |\widehat{\rho}(\epsilon s)| ds d\xi d\tau \\
					&\lesssim_N 2^{-N\max(j,k)} \iiint |\widehat{P_j(f\mu)}(\xi)| |\widehat{P_k(g\mu)}(\tau)| 
					(1+|s|)^{-2}  ds d\xi d\tau \\
					&= 2^{-N\max(j,k)} \int |\widehat{P_j(f\mu)}(\xi)| d\xi 
					\int |\widehat{P_k(g\mu)}(\tau)|d\tau \int (1 + |s|)^{-2} ds.
				\end{align*}
				The integral in $s$ is finite. By Cauchy-Schwarz and Lemma \ref{l:wolff_lemma}, the integrals in $\xi$ and $\tau$ are dominated by 
				$2^{jd/2}2^{j(d-s)/2}$ and $2^{kd/2}2^{k(d-s)/2}$ respectively. Thus 
				\begin{equation*}
					|\ang{\cR_t^\epsilon(P_j(f\mu)), P_k(g\mu)}| \lesssim_N 2^{(-N + 2d - s)\max(j,k)},
				\end{equation*}
				so by choosing a large enough $N$ we are done. 
				
				Now we prove $\eqref{eq:diag:meat}$.
				We have 
				\begin{equation*}
					\nabla_x \Psi_{\xi,\tau,s} = -\tau + sy
				\end{equation*}
				and
				\begin{equation*}
					\nabla_y \Psi_{\xi,\tau,s} = \xi + sx.
				\end{equation*}
				Assume without loss of generality that $j > k +K$, so $|\xi| \ll |\tau|$.
				When $|s| \ll |\tau|$,
				\begin{equation*}
					|\nabla_x \Psi_{\xi, \tau, s}| \gtrsim |\tau| - |s| \gtrsim |\tau|,
				\end{equation*}
				Where we used that $\eta$ has fixed compact support. 
				If $|s| \gg |\tau|$,
				\begin{equation*}
					|\nabla_x \Psi_{\xi, \tau, s}| \gtrsim |s| - |\tau| \gtrsim |s|,
				\end{equation*}
				where we use that $\eta$ is not supported near the origin. 
				If $|s| \approx |\tau|$,
				\begin{align*}
					|\nabla_y \Psi_{\xi,\tau, s}| &\gtrsim |s| - |\xi| \\
					&\approx |\tau| - |\xi| \\
					&\gtrsim |\tau|.
				\end{align*}
				In any case $|\nabla \Psi_{\xi,\tau,s}| \gtrsim \max(|\tau|, |s|)$.
				
				It is immediate that all the partials of $\Psi_{\xi, \tau, s}$ are bounded above by a constant multiple of $\max(|\tau|, |s|)$. 
				Consider the differential operator 
				\begin{equation*}
					L = \frac{1}{2\pi i} \frac{\nabla \Psi_{\xi, \tau, s}}{|\nabla \Psi_{\xi, \tau, s}|^2} \cdot \nabla,
				\end{equation*}
				for which $e^{2\pi i \Psi_{\xi, \tau, s}}$ is clearly an eigenvalue. Therefore $L^N(e^{2\pi i \Psi_{\xi, \tau, s}}) = e^{2\pi i \Psi_{\xi, \tau, s}}$ for any positive integer $N$. Thus 
				\begin{equation*}
					I_{jk}(\xi, \tau, s) = \iint L^N(e^{2\pi i \Psi_{\xi, \tau, s}}) \eta dy'dx = \iint e^{2\pi i \Psi_{\xi,\tau, s}} (L^t)^N(\eta) dy' dx.
				\end{equation*}
				The transpose $L^t$ of $L$ is given by 
				\begin{align*}
					L^t(f) &= -\frac{1}{2\pi i} \nabla \cdot \left(\frac{\nabla \Psi_{\xi, \tau, s}}{|\nabla \Psi_{\xi, \tau, s}|^2} f \right) \\
					&= - \frac{1}{2\pi i} \left( \nabla f \cdot \frac{\nabla \Psi_{\xi, \tau, s}}
					{|\nabla \Psi_{\xi, \tau, s}|^2}  - f \cdot \frac{\nabla^2 \Psi_{\xi, \tau, s}}{|\nabla \Psi_{\xi, \tau, s}|^2}\right)
				\end{align*}
				and taking the modulus, 
				\begin{equation*}
					|L^t(\eta)| \lesssim |\nabla \eta|  \max(|\tau|, |s|)^{-1} + |\eta| 
					\max(|\tau|, |s|)^{-1}.
				\end{equation*}
				We can continue integrating by parts up to any positive $N'$ and obtain a bound
				\begin{equation*}
					|(L^t)^{N'}(\eta)| \lesssim_{N'} \max(|\tau|, |s|)^{-N'} \sum_{|\alpha| \leq N'} |\partial^\alpha \eta|.
				\end{equation*}
				Taking the modulus of $I_{jk}(\xi, \tau, s)$, we obtain 
				\begin{equation*}
					|I_{jk}(\xi,\tau, s)| \lesssim_{N'} \max(|\tau|, |s|)^{-N'} \iint \sum_{|\alpha| \leq N'} |\partial^\alpha \eta| dy'dx \lesssim_N \max(|\tau|, |s|)^{-N'},
				\end{equation*}
				since $\eta$ is Schwarz. Since $|\tau| \geq 1$,
				\begin{align*}
					|I_{jk}(\xi, \tau, s)| &\lesssim_N \max(|\tau|, |s|)^{-2} \max(|\tau|, |s|)^{-N' + 2} \\
					&\lesssim (1+|s|)^{-2} |\tau|^{-N' + 2} \\
					&\lesssim (1 + |s|)^{-2} 2^{-(N' + 2)\max(j,k)}.
				\end{align*}
				Taking $N'$ large enough we are done.
			\end{proof}
			With Lemma \ref{l:diag} proved, we are done.
		\end{proof}
	
	\section{acknowledgements}
		I would like to thank my undergraduate mentor and friend Alex Iosevich, for his support in this project and in many others. 
		Thank you to Ben Baily, Brian Hu, and Ethan Pesikoff for
		helping to develop the symmetric tree cover idea over SMALL 2021.

\bibliography{bib}{}

\begin{thebibliography}{10}

\bibitem{chains}
Michael Bennett, Alexander Iosevich, and Krystal Taylor.
\newblock Finite chains inside thin subsets of {$\mathbb{R}^d$}.
\newblock {\em Analysis \& PDE}, 9(3):597--614, 2016.

\bibitem{du2021weighted}
Xiumin Du, Larry Guth, Yumeng Ou, Hong Wang, Bobby Wilson, and Ruixiang Zhang.
\newblock Weighted restriction estimates and application to falconer distance
  set problem.
\newblock {\em American Journal of Mathematics}, 143(1):175--211, 2021.

\bibitem{du2021improved}
Xiumin Du, Alex Iosevich, Yumeng Ou, Hong Wang, and Ruixiang Zhang.
\newblock An improved result for falconer’s distance set problem in even
  dimensions.
\newblock {\em Mathematische Annalen}, 380(3):1215--1231, 2021.

\bibitem{du2019sharp}
Xiumin Du and Ruixiang Zhang.
\newblock Sharp {$ L^{2}$} estimates of the schr{\"o}dinger maximal function in
  higher dimensions.
\newblock {\em Annals of Mathematics}, 189(3):837--861, 2019.

\bibitem{falconer1985}
Kenneth~J Falconer.
\newblock On the hausdorff dimensions of distance sets.
\newblock {\em Mathematika}, 32(2):206--212, 1985.

\bibitem{group-actions}
Allan Greenleaf, Alex Iosevich, Bochen Liu, and Eyvindur Palsson.
\newblock A group-theoretic viewpoint on erd{\H{o}}s--falconer problems and the
  mattila integral.
\newblock {\em Revista Matem{\'a}tica Iberoamericana}, 31(3):799--810, 2015.

\bibitem{nonempty_interior_radon}
Allan Greenleaf, Alex Iosevich, and Krystal Taylor.
\newblock Configuration sets with nonempty interior.
\newblock {\em The Journal of Geometric Analysis}, 31(7):6662--6680, 2021.

\bibitem{guth2020falconer}
Larry Guth, Alex Iosevich, Yumeng Ou, and Hong Wang.
\newblock On falconer’s distance set problem in the plane.
\newblock {\em Inventiones mathematicae}, 219(3):779--830, 2020.

\bibitem{guth-katz}
Larry Guth and Nets~Hawk Katz.
\newblock On the erd{\H{o}}s distinct distances problem in the plane.
\newblock {\em Annals of mathematics}, pages 155--190, 2015.

\bibitem{iosevich2019maximal}
A~Iosevich, Ben Krause, E~Sawyer, K~Taylor, and I~Uriarte-Tuero.
\newblock Maximal operators: scales, curvature and the fractal dimension.
\newblock {\em Analysis Mathematica}, 45(1):63--86, 2019.

\bibitem{treesIosevichTaylor}
A~Iosevich and K~Taylor.
\newblock Finite trees inside thin subsets of rd, modern methods in operator
  theory and harmonic analysis, 51--56.
\newblock {\em Springer Proc. Math. Stat}, 291, 2019.

\bibitem{IosevichTaylorIgnacio}
Alex Iosevich, Krystal Taylor, and Ignacio Uriarte-Tuero.
\newblock Pinned geometric configurations in euclidean space and riemannian
  manifolds, 2016.

\bibitem{mattila_geometry}
Pertti Mattila.
\newblock {\em Geometry of sets and measures in Euclidean spaces: fractals and
  rectifiability}.
\newblock Number~44. Cambridge university press, 1999.

\bibitem{PhongSteinRadon}
Duong~H Phong and Elias~M Stein.
\newblock Radon transforms and torsion.
\newblock {\em International Mathematics Research Notices}, 1991(4):49--60,
  1991.

\bibitem{steinshakfunctional}
Elias~M Stein and Rami Shakarchi.
\newblock Functional analysis: introduction to further topics in analysis.
\newblock {\em Princeton Lectures in Analysis}, 4, 2011.

\bibitem{StrichartzFourierAsymptotics}
Robert~S. Strichartz.
\newblock Fourier asymptotics of fractal measures.
\newblock {\em Journal of Functional Analysis}, 89(1):154--187, 1990.

\bibitem{wolff2003lectures}
Thomas~H Wolff.
\newblock {\em Lectures on harmonic analysis}, volume~29.
\newblock American Mathematical Soc., 2003.

\end{thebibliography}
\bibliographystyle{plain}

\end{document}